\documentclass[12pt]{article}

\usepackage{amsmath,amssymb,amsbsy,amsfonts,amsthm,latexsym,
                    amsopn,amstext,amsxtra,euscript,amscd}

\begin{document}

\newtheorem{theorem}{Theorem}
\newtheorem{lemma}[theorem]{Lemma}
\newtheorem{claim}[theorem]{Claim}
\newtheorem{cor}[theorem]{Corollary}
\newtheorem{prop}[theorem]{Proposition}
\newtheorem{definition}{Definition}
\newtheorem{question}[theorem]{Question}

\def\cA{{\mathcal A}}
\def\cB{{\mathcal B}}
\def\cC{{\mathcal C}}
\def\cD{{\mathcal D}}
\def\cE{{\mathcal E}}
\def\cF{{\mathcal F}}
\def\cG{{\mathcal G}}
\def\cH{{\mathcal H}}
\def\cI{{\mathcal I}}
\def\cJ{{\mathcal J}}
\def\cK{{\mathcal K}}
\def\cL{{\mathcal L}}
\def\cM{{\mathcal M}}
\def\cN{{\mathcal N}}
\def\cO{{\mathcal O}}
\def\cP{{\mathcal P}}
\def\cQ{{\mathcal Q}}
\def\cR{{\mathcal R}}
\def\cS{{\mathcal S}}
\def\cT{{\mathcal T}}
\def\cU{{\mathcal U}}
\def\cV{{\mathcal V}}
\def\cW{{\mathcal W}}
\def\cX{{\mathcal X}}
\def\cY{{\mathcal Y}}
\def\cZ{{\mathcal Z}}

\def\A{{\mathbb A}}
\def\B{{\mathbb B}}
\def\C{{\mathbb C}}
\def\D{{\mathbb D}}
\def\E{{\mathbb E}}
\def\F{{\mathbb F}}
\def\G{{\mathbb G}}
\def\I{{\mathbb I}}
\def\J{{\mathbb J}}
\def\K{{\mathbb K}}
\def\L{{\mathbb L}}
\def\M{{\mathbb M}}
\def\N{{\mathbb N}}
\def\O{{\mathbb O}}
\def\P{{\mathbb P}}
\def\Q{{\mathbb Q}}
\def\R{{\mathbb R}}
\def\S{{\mathbb S}}
\def\T{{\mathbb T}}
\def\U{{\mathbb U}}
\def\V{{\mathbb V}}
\def\W{{\mathbb W}}
\def\X{{\mathbb X}}
\def\Y{{\mathbb Y}}
\def\Z{{\mathbb Z}}

\def\ep{{\mathbf{e}}_p}
\def\em{{\mathbf{e}}_m}

\def\scr{\scriptstyle}
\def\\{\cr}
\def\({\left(}
\def\){\right)}
\def\[{\left[}
\def\]{\right]}
\def\<{\langle}
\def\>{\rangle}
\def\fl#1{\left\lfloor#1\right\rfloor}
\def\rf#1{\left\lceil#1\right\rceil}
\def\le{\leqslant}
\def\ge{\geqslant}
\def\eps{\varepsilon}
\def\mand{\qquad\mbox{and}\qquad}

\def\vec#1{\mathbf{#1}}
\def\inv#1{\overline{#1}}
\def\vol#1{\mathrm{vol}\,{#1}}
\def\dist{\mathrm{dist}}

\def\fA{{\mathfrak A}}
\def\fU{{\mathfrak U}}
\def\fV{{\mathfrak V}}

\def\GL{\mathrm{GL}}
\def\SL{\mathrm{SL}}

\def\Hba{\overline{\cH}_{a,m}}
\def\Hta{\widetilde{\cH}_{a,m}}
\def\Hb1{\overline{\cH}_{m}}
\def\Ht1{\widetilde{\cH}_{m}}

\def\Zm{\Z/m\Z}

\def\Err{{\mathbf{E}}}

\newcommand{\comm}[1]{\marginpar{%
\vskip-\baselineskip 
\raggedright\footnotesize
\itshape\hrule\smallskip#1\par\smallskip\hrule}}

\def\xxx{\vskip5pt\hrule\vskip5pt}


\title{\bf 
Modular Hyperbolas}

\author{
{\sc Igor E. Shparlinski} \\
{Department of Computing, Macquarie University} \\
{Sydney, NSW 2109, Australia} \\
{igor.shparlinski@mq.edu.au}}

\date{\today}
\pagenumbering{arabic}

\maketitle

\begin{abstract}
We give a survey of a variety of recent results
about the distribution and some geometric properties
of points $(x,y)$ on modular hyperbolas
$xy \equiv a \pmod m$. We also  outline a very diverse range of
applications of such results,  discuss multivariate
generalisations and suggest a number of   open problems
of different levels of difficulty. 
\end{abstract}

\section{Introduction}
\label{sec:intro}

\subsection{Modular Hyperbolas}

For a positive integer  $m$ and an arbitrary integer $a$ with $\gcd(a,m) =1$,
we consider the  set of points $(x,y)$ on the  modular hyperbola
$$
\cH_{a,m} = \{(x,y) \ : \  xy \equiv a \pmod m\}.
$$
We give a survey of various results about the distribution and
some geometric properties of points
on $\cH_{a,m}$. We also briefly consider 
some multidimensional
generalisations, which often require   very different techniques.
Several  open problems are formulated as well. Our main goal is 
to show that although $\cH_{a,m}$ is defined by one of the simplest 
possible polynomial congruences, it exhibits many mysterious
properties and  
very surprising links with a wide variety of 
classical number theoretic questions and beyond.

In this survey, we do not present  complete proofs
but rather explain
their underlying ideas and specific ingredients.
Typically the error terms in the asymptotic formulas
we give contain a factor $m^{o(1)}$. In most cases it
can be replaced by a more explicit function. Furthermore,
when $m$ is prime it can usually (but not always) be replaced by
just some low power of $\log m$.

There is a large number of
papers in this area which rather routinely study seemingly distinct, but  
in fact closely related,
problems about $\cH_{a,m}$ on a case by case basis.
Here,  we  explain some standard principles which can be
used to derive these and many other results of similar spirit
about the points
on
$\cH_{a,m}$ as simple corollaries of just one general
result
about the uniformity of
distribution of  points on  $\cH_{a,m}$  in certain domains.
In Section~\ref{sec:Indiv}, such a  result is presented in
 Theorem~\ref{thm:UD Gen} and derived
in a very straightforward fashion from bounds
of Kloosterman sums (see~\eqref{eq:Kloost bound} below)
using some standard arguments. Results of this type are quite 
standard and can be obtained for many other curves (at least in
the case of prime modulus $m$, where the Bombieri bound~\cite{Bomb}
provides a readily available substitute for~\eqref{eq:Kloost bound}).  

However,  most of the other 
results on modular hyperbolas rely on some rather subtle number theoretic arguments 
which in general do not apply to other polynomial congruences. 
This places modular hyperbolas in a very special position 
and shows that they define a mathematically much richer structure than 
a ``typical'' polynomial congruence. 

\subsection{Distribution of Points  and Kloosterman Sums}

Since  the distribution of points on  $\cH_{a,m}$ 
is our primal object of study, we introduce the following 
definition.
Given two sets of integers $\cX$ and $\cY$, we
write 
$$
\cH_{a,m}(\cX, \cY) = \{(x,y) \in \cH_{a,m} \ : \  x \in \cX, \ y  \in \cY \}.
$$
Since the case of $a=1$ is of special interest we also write
$$
\cH_m = \cH_{1,m}, \mand \cH_{m}(\cX, \cY) = \cH_{1,m}(\cX, \cY).
$$

Obtaining precise asymptotic formulas for and establishing the positivity
of $\#\cH_{a,m}(\cX, \cY)$ for various ``interesting'' sets  $\cX$ and
$\cY$ have been the central themes of many works in this direction.
Certainly the problem becomes harder and far more interesting when the
sets  $\cX$ and $\cY$ become ``thinner''.

Let
$$
\em( z) = \exp( 2 \pi i z/m).
$$

One immediately observes that the well-known bound
\begin{equation}
\label{eq:Kloost bound}
|K_m(r,s)|
\le \(m \gcd(r,s,m)\)^{1/2+o(1)},
\end{equation} of 
{\it Kloosterman sums\/}
\begin{equation}
\label{eq:Kloost}
K_m(r,s) = \sum_{\substack{(x,y) \in \cH_{m} \\ 1 \le x,y \le m}}
\em\(rx + sy\),
\end{equation}
see~\cite[Corollary~11.12]{IwKow},
   can be used to study  the points on $ \cH_{a,m}$.
One only needs to recall some standard tools
which link exponential sums with
uniformity of distribution which we present in
Section~\ref{sec:UD tools}.

We remark that most of the  results which rely  only
on~\eqref{eq:Kloost bound}
do not appeal to anything specific about the congruence $xy \equiv a \pmod m$,
and at least when $m$ prime they can be extended to the distribution of
solutions to more general congruences $f(x,y) \equiv 0 \pmod m$, with
a polynomial $f$ with integer coefficients. In the case of prime $m$,
the Bombieri bound~\cite{Bomb} of exponential sums along a curve
replaces the bound~\eqref{eq:Kloost bound}. Although, as we have mentioned,
we present such a generic result in Section~\ref{sec:Indiv}, our
main purpose is to outline some more  intricate arguments, which use
special properties of the congruence $xy \equiv a \pmod m$ and cannot be
generalised to other congruences. In particular,
in Section~\ref{sec:Aver} we discuss the behaviour
of points on $\cH_{a,m}$ on average over $a$. We also
describe   some geometric properties of  the set  $\cH_{a,m}$
in Sections~\ref{sec:Dist}--\ref{sec:Visible}, 
show that it cannot be too concentrated even in very
small squares in Section~\ref{sec:Cocentr} and discuss
some arithmetic properties of elements of $\cH_{a,m}$
in Section~\ref{sec:Arith Fun}.

Before presenting results about the properties
of $\cH_{a,m}$,  we give a short overview of
the surprisingly diverse variety
of number theoretic tools which have been used in the study
of $\cH_{a,m}$ and their multivariate generalisations,
see Sections~\ref{sec:Exp}--\ref{sec:ANT}.

Finally, we demonstrate the wealth and diversity 
of various applications of the results on the
distribution of points on $\cH_{a,m}$. Some of these 
applications  are quite
natural with very transparent connections to $\cH_{a,m}$, see
Section~\ref{sec:Lehm}. However, there are also
several   less obvious and thus  much 
more exciting applications, 
see Sections~\ref{sec:Angles}--\ref{sec:Dlog Fact}.
An especially  striking example  of such unexpected
applications is given by a result  of~\cite{Merel} on
torsions of elliptic curves, see Section~\ref{sec:Ell Tors}.

\subsection{Notation}

Throughout the paper, any implied constants in symbols $O$, $\ll$
and $\gg$ may occasionally depend, where obvious, on the real positive
parameter  $\eps$ and are absolute otherwise. We recall
that the notations $U = O(V)$,  $U \ll V$ and  $V \gg U$  are
all equivalent to the statement that $|U| \le c V$ holds 
with some constant $c> 0$.

We use $p$, with or without a subscript,
to denote a prime number and use $m$ to denote a positive integer.

We denote by $\Zm$  the residue ring modulo $m$.
Typically we assume that  the
set $\{0, \ldots, m-1\} $ is used
to represent the elements of $\Zm$.
Accordingly, we often consider the following subset of $\cH_{a,m}$
$$
\Hba  = \cH_{a,m}\cap [0,m-1]^2 .
$$
We also put
$$
\Hb1  = \cH_{m}\cap [0,m-1]^2.
$$

We always follow the convention that  arithmetic operations in the 
arguments of $\em$ are performed modulo $m$. In particular,
the Kloosterman sums~\eqref{eq:Kloost} can 
also be written as 
$$
K_m(r,s) = \sum_{\substack{x=1 \\ \gcd(x,m)=1}}^m
\em\(rx + sx^{-1}\).
$$

As usual,  $\omega(k)$, $\tau(k)$ and
$\varphi(k)$ denote the number of distinct prime divisors,
the number of positive integer divisors and
the Euler function of $k\ge 1$, respectively.

Finally,   $\mu(k)$ denotes the M\"obius function.
  We recall that $\mu(1) = 1$, $\mu(k) = 0$ if $k \ge 2$ is not
squarefree   and $\mu(k) = (-1)^{\omega(k)}$ otherwise.

\subsection{Acknowledgements}

Thanks go to Mizan Khan who introduced
the beautiful and mysterious world of modular 
hyperbolas  to the author.

The author would also like to thank Alexey Ustinov, Arne Winterhof 
and the referee, 
for a careful reading of the
manuscript and making many valuable suggestions.

This work was supported in part by ARC grant DP1092835.

\section{Number Theory Background}

\subsection{Exponential and Character Sums}
\label{sec:Exp}

We have already mentioned the prominent role of
Kloosterman sums~\eqref{eq:Kloost}
and the bound~\eqref{eq:Kloost bound}
in particular.

Most of the works also use the identity
\begin{equation}
\label{eq:Ident add char}
\frac{1}{m}\sum_{r \in \Zm} \em(rv) =
\left\{\begin{array}{ll}
1,&\quad\text{if $v\equiv 0 \pmod m$,}\\
0,&\quad\text{if $v\not\equiv 0 \pmod m$,}
\end{array}
\right.
\end{equation}
to express various characteristic
functions via exponential sums. 
Thus, we relate various counting questions
to exponential sums.

It is very often complemented by the bound
\begin{equation}
\label{eq:Incompl}
   \sum_{z=W+1}^{W+Z} \em(rz) \ll \min\{Z, m/|r|\}
\end{equation}
which holds for any integers  $r$, $W$ and $Z\ge 1$ with $0 < |r| \le m/2$,
see~\cite[Bound~(8.6)]{IwKow}.

We now recall  the estimate from~\cite{Shp1}
of exponential sums with rational functions
of special type, which generalises the
bound~\eqref{eq:Kloost bound} of 
Kloosterman sums~\eqref{eq:Kloost}.

\begin{lemma}
\label{lem:Exp Sparse}
Let $n_1, \ldots, n_s$ be $s \ge 2$ nonzero fixed pairwise distinct
integers. Then the bound
$$
\max_{\gcd(a_1, \ldots, a_s,m)=d}
\left|\sum_{\substack{z=1\\ \gcd(z,m) = 1}}^m
\em(a_1 z^{n_1} + \ldots + a_s z^{n_s})\right| \le d^{1/s} m^{1 - 1/s + o(1)}
$$
  holds.
\end{lemma}

We recall that several more bounds of exponential 
sums with sparse polynomials of large degree
are also given in~\cite{Bour0,Bour3,CPR}.

However, in many cases using bounds
of  multiplicative character sums yields
stronger results.

Let $\varPhi_m$ be the set of all $\varphi(m)$
multiplicative characters modulo  $m$.
   We have the following analogue
of~\eqref{eq:Ident add char}. For any integer $r$,
\begin{equation}
\label{eq:Ident mult char}
\frac{1}{\varphi(m)}
   \sum_{\chi \in \varPhi_m}\chi\(r\)
=\left\{\begin{array}{rll}1 &\text{if}\ r
\equiv 1 \pmod m, \\ 0 &\text{otherwise.}\end{array}\right.
\end{equation}

We also use $\chi_0$ to denote the principal character.

The following result  is a combination of  the
P{\'o}lya-Vinogradov  (for $\nu =1$) and Burgess
(for $\nu\ge2$)  bounds,
see~\cite[Theorems~12.5 and 12.6]{IwKow}.

\begin{lemma}
   \label{lem:PVB} For arbitrary  integers
$W$ and $Z$ with  $1 \le Z \le m$, the bound
$$
\max_{\substack{\chi \in \varPhi_m\\ \chi \ne \chi_0}}
\left| \sum_{z = W+1}^{W+Z}
\chi(z)\right|  \le Z^{1 -1/\nu} m^{(\nu+1)/4\nu^2 + o(1)}
$$
holds with $\nu = 1,2,3$ for any $m$ and with an arbitrary
positive integer $\nu$ if $m =p$ is a prime.
\end{lemma}

The identity~\eqref{eq:Ident mult char} immediately implies
that for $1 \le Z \le m$
\begin{equation}
\label{eq:2nd Moment}
\sum_{\chi \in \varPhi_m}
\left| \sum_{z = W+1}^{W+Z} \chi(z)\right|^2 = \varphi(m)
\sum_{\substack{z =W+1 \\ \gcd(z, m)=1}}^{W+Z} 1 \le \varphi(m) Z
\end{equation}
which has been used in many works on $\cH_{a,m}$.

Furthermore, it turns out, that  sometimes one gets better results using the
following fourth moment estimate from~\cite{ACZ}
(for prime $m=p$) and~\cite{FrIw2} 
(for arbitrary $m$), see also~\cite{GarGar}. 
In fact, these results have recently  been 
generalised in~\cite{CochSih}, which we present in the 
following slightly  less precise form.

\begin{lemma}
   \label{lem:4th Moment} For  arbitrary integers
$W$,  and $Z \le m$, the bound
$$
\sum_{\substack{\chi \in \varPhi_m\\ \chi \ne \chi_0}}
\left| \sum_{z = W+1}^{W+Z} \chi(z)\right|^4
\le   m^{1 + o(1)} Z^2
$$
holds.
\end{lemma}

Note that many of the results below are based on previous estimates 
of~\cite{ACZ} (which require $m=p$ to be prime) 
and~\cite{FrIw2} (which apply to any $m$ but require $W = 0$). 
So now Lemma~\ref{lem:4th Moment} allows us  to drop these
restrictions.

For example, bounds of higher moments of multiplicative character sums, 
which in particular are based on 
combining Lemma~\ref{lem:PVB} and a previous version of 
Lemma~\ref{lem:4th Moment} with some
other arguments,  
have been given~\cite{CoZh}, and can now be generalised.

In~\cite{Gar3,Gar6}  a new argument has been introduced to 
this area, which is based on a 
bound of~\cite{Hux} on the number of large 
values of Dirichlet polynomials,
see also the original 
papers~\cite{HuxJut,Jut}  as well as~\cite[Chapter~9]{IwKow} for 
some other estimates for  Dirichlet polynomials. 
This line of research is extremely interesting
and definitely deserves more investigation. 

\subsection{Average Values of $L$-functions}
\label{sec:L-fun}

Many quantities related to the distribution 
of points on $\cH_{a,m}$ on average over $a$,
lead to studying various average values of $L$-functions 
$L(1, \chi)$
with multiplicative characters $\chi\in \varPhi_m$, 
see, for 
example,~\cite{LiuZha1,LiuZha2,LiuZha3,LiuZha4,LiuZha5,LuYi2,Zha6}. 
Sometimes such sums are weighted by character sums and 
are of independent interest.

For example, let $R_{a,m}(N)$ be the number of points
$(x,y) \in \cH_{a,m}$ with $1 \le x \le N$ and $1 \le y < m$,
and such that $x+y$ is odd.
It is shown in~\cite{LuYi2}, that for a prime $p$ the second 
moment of the differences $R_{a,p}(N) - N/2$, $a=1, \ldots, p-1$, 
can be expressed via the sums
\begin{eqnarray*}
\sigma_1(p,N) & = &\sum_{\substack{\chi \in \varPhi_p\\ \chi(-1) =-1}} \left|\sum_{n=1}^N
  (-1)^n \chi(n)\right|^2
|L(1, \chi)|^2, \\
\sigma_2(p,N) & = &\sum_{\substack{\chi \in \varPhi_p\\ \chi(-1) =-1}}\chi(2) \left|\sum_{n=1}^N
(-1)^n \chi(n)\right|^2 |L(1, \chi)|^2 .
\end{eqnarray*}
Using a combination of the methods of~\cite{Shp10,Zha7}, 
in~\cite{LuYi2} the asymptotic formulas 
\begin{eqnarray*}
\sigma_1(p,N) & = &\alpha pN + O\(N^2 p^{o(1)} + p (\log N)^2\), \\
\sigma_2(p,N) & = &\frac{\alpha }{2}pN + O\(N^2 p^{o(1)} + p (\log N)^2\).
\end{eqnarray*}
are  given, where
$$
\alpha = \frac{\pi^2}{12}\(1 + 
\frac{16}{9\zeta(3)} \sum_{k=1}^\infty \frac{1}{(2k+1)^2} 
\sum_{h=0}^k \frac{1}{2h+1}\)
$$
and $\zeta(s)$ is the Riemann zeta-function.  
Clearly, the above formulas are nontrivial for $p^\varepsilon < N \le 
p^{1-\varepsilon}$ for any fixed $\varepsilon>0$. On the
other hand, 
$$\sigma_1(p,p)=\sigma_2(p,p) = 0, 
$$ 
so there is  some kind of the ``phase-transition'' area when $N$ gets 
close to $p$ which would be interesting to understand  more.

\subsection{Theory of Uniform Distribution}
\label{sec:UD tools}

For a finite set $\cF \subseteq [0,1]^s$ of the $s$-dimensional
unit cube, we define its {\it discrepancy
with respect to a domain $\varXi \subseteq [0,1]^s$\/}
as
$$
\varDelta(\cF ,\varXi) = \left| \frac{ \#\{ \vec{f}\in\cF :\ \vec{f}\in
\varXi\} }{\#\cF} - \lambda(\varXi)\right|,
$$
where $\lambda$ is the Lebesgue measure on $ [0,1]^s$.

We now define  the {\it  discrepancy\/} of $\cF $ as
$$
D(\cF) = \sup_{\varPi  \subseteq [0,1]^s}  \varDelta(\cF ,\varPi) ,
$$
where the supremum is taken over all boxes $\varPi = [\alpha_1,
\beta_1) \times \ldots \times [\alpha_s, \beta_s) \subseteq [0,1]^s$.

A link between the discrepancy and exponential sums is provided by
the celebrated {\it Koksma--Sz\"usz inequality\/}, see~\cite[Theorem~1.21]{DrTi}.
However, for points of 
$\cH_{a,m}$, due to  the discrete structure of the problem, one can immediately
establish such a link directly  by the identity~\eqref{eq:Ident add char}.

For example, one can consider the points
$$
\(\frac{x}{m}, \frac{y}{m}\) \in [0,1]^2,  \qquad (x,y) \in \Hba,
$$
and apply the bound~\eqref{eq:Kloost bound} to estimate their
discrepancy,
which in turn is equivalent to studying points of
$\cH_{a,m}(\cX,\cY)$ where $\cX$ and $\cY$ are sets of consecutive integers.

Moreover, the {\it Koksma--Hlawka inequality\/}, see~\cite[Theorem~1.14]{DrTi},
allows us to estimate average values of various functions on the points
$(x,y) \in \cH_{a,m}$.

\begin{lemma}
\label{lem:K-H} For any continuous function $\psi(\vec{z})$ on the
unit cube $[0,1]^s$
and a finite set $\cF \subseteq [0,1]^s$  
of discrepancy  $D(\cF)$, 
the following bound holds:
$$
    \frac{1}{\# \cF } \sum_{\vec{f} \in \cF} \psi(\vec{f}) =
\int_{[0,1]^s} \psi(\vec{z})\, d \vec{z}
   + O\(D(\cF) \)
$$
where the implied constant depends only on $s$ and the function $\psi$.
\end{lemma}

To study $\cH_{a,m} \cap \cW$  for more general sets $\cW$ some additional
tools are required from the theory of uniform distribution.

As usual, we define the  distance between a vector $\vec{u} \in [0,1]^s$
and a set $\varXi\subseteq [0,1]^s $  by
$$
\dist(\vec{u},\varXi) = \inf_{\vec{w} \in\varXi}
\|\vec{u} - \vec{w}\|,
$$
where  $\|\vec{v}\|$ denotes the Euclidean norm of $\vec{v}$. Given
$\varepsilon >0$ and a domain  $\varXi \subseteq [0,1]^s $ we define
the  sets
$$
\varXi_\varepsilon^{+} = \left\{ \vec{u} \in  [0,1]^s \backslash
\varXi \ : \ \dist(\vec{u},\varXi) < \varepsilon \right\}
$$
and
$$
\varXi_\varepsilon^{-} = \left\{ \vec{u} \in \varXi \ : \
\dist(\vec{u},[0,1]^s \backslash \varXi )  < \varepsilon  \right\} .
$$

Let $h(\varepsilon)$ be an arbitrary increasing function defined for
$\varepsilon >0$ and such that
$$
\lim_{ \varepsilon \to 0}h(\varepsilon) = 0.
$$
As in~\cite{Lac,NiWi}, we define the
class $\cS_h$ of  domains  $\varXi \subseteq [0,1]^s $ for which
$$
\lambda\(\varXi_\varepsilon^{+} \)\le h(\varepsilon)
\qquad \mbox{and}
\qquad
\lambda\(\varXi_\varepsilon^{-} \)\le h(\varepsilon) 
$$
for any $\varepsilon >0$.

A relation  between $D(\cF)$ and $\varDelta(\cF ,\varXi)$
for $\varXi \in \cS_h$ is given by the following inequality of~\cite{Lac}
(see also~\cite{NiWi}).

\begin{lemma}
   \label{lem:LNW bound} For any domain  $\varXi \in \cS_h$, we have
$$
\varDelta(\cF ,\varXi) \ll h\(s^{1/2} D(\cF )^{1/s}\)  .
$$
\end{lemma}

Finally,  the following bound, which is a special case
of a more general result of H.~Weyl~\cite{Weyl}
shows that if  $\varXi$ has a piecewise smooth boundary
then $\varXi \in \cS_h$ for some linear function
$h(\varepsilon) = C\varepsilon$.

\begin{lemma}
   \label{lem:Weyl bound} For any domain  $\varXi \in [0,1]^s$ with
a piecewise smooth boundary, we
have
$$
\lambda\(\varXi_\varepsilon^{\pm }\) = O(\varepsilon).
$$
\end{lemma}

To use the above results for the  study of points on $\cH_{a,m}$, one usually
considers points
\begin{equation}
\label{eq:Map}
\(\frac{x}{m},\frac{y}{m}\) \in [0,1]^2, \qquad  (x,y) \in \cH_{a,m},
\ 1 \le x,y \le m.
\end{equation}

\subsection{Arithmetic Functions, Divisors, Prime
Numbers}
\label{sec:ANT}

Certainly some elementary bounds such
as
$$
\varphi(k) \gg \frac{k}{\log \log (k+2)}
$$
and 
\begin{equation}
\label{eq:tau}
2^{\omega(k)} \le \tau(k) \le \exp\(\(\log 2+ o(1)\)\frac{\log k}{\log \log k}\),
\end{equation}
see~\cite[Section~I.5.2 and~I.5.4]{Ten}, 
appear at various stages of the proofs of relevant results.

The following well-known consequence of the {\it sieve of
Eratosthenes\/}  (essentially
of the inclusion-exclusion principle expressed
via the M\"obius function) is very often needed to estimate
the main terms of various asymptotic formulas (see, for
example,~\cite{Shp6,Shp11}).

\begin{lemma}
\label{lem:erat}
For any integers $m,  Z \ge 1$ and $W \ge 0$,
$$ \sum_{\substack{z =W+1 \\ \gcd(z, m)=1}}^{W+Z} 1
= \frac{\varphi (m) }{m}Z + O(2^{\omega(m)}).
$$
\end{lemma}

For an infinite monotonically increasing sequence of positive integers
$\cA=\(a_n\)_{n=1}^\infty$, we define
$$
H(x,y,z;\cA) = \#\{ n\le x \ :\  \exists \, d|a_n \text{ with } y<d\le
z\}. $$
For $\cA=\N$, the set of natural numbers,
   the order of magnitude of $H(x,y,z;\N)$ for all
$x,y,z$ has been determined in~\cite{Ford}, see also~\cite{HalTen}.
   Also in~\cite{Ford}, one can find upper bounds for $H(x,y,z;\cP_b)$ of
the expected order of
     magnitude, where $\cP_b=\{ p+b \ :\  p\text{ prime}\}$ is a set of
so-called shifted
     primes.
However, for the problem of studying  $\cH_{a,m}$, we need
analogous
     results where $n$ is restricted to an arithmetic progression.
More precisely, let us define the sequences
$$ \cT_k=\{ mk-1\ :\  m\in \N\} \mand\cU_k = \{
pk-1 \ :\  p\text{ prime}\}. $$
It has been shown in~\cite{FKSY} that the arguments of~\cite{Ford}
imply the following estimates.

It is usual that in questions of this kind, the constant
\begin{eqnarray}
\label{eq:kappa}
\kappa = 1 - \frac{1+ \log \log 2}{\log 2}
= 0.086071\ldots \,.
\end{eqnarray}
plays an important role, see also~\cite{HalTen}.

\begin{lemma}
\label{lem:div int}
Uniformly for $100 \le y\le x^{0.51}$, $1.1y\le z\le y^{1.1}$, $1\le
k\le \log x$, we have
\begin{eqnarray*}
H(x,y,z;\cT_k) & \ll & x \frac{k}{\varphi(k)} u^\kappa (\log
(1/u))^{-3/2}, \\
H(x,y,z;\cU_k) & \ll & x \frac{k}{\varphi(k)} u^\kappa (\log
(1/u))^{-3/2},
\end{eqnarray*}
where $z=y^{1+u}$.
\end{lemma}

A certain  result of~\cite{FKSY} relies on  the existence
of infinitely many  primes $p$ with a prescribed structure of
divisors of $p-1$, which is done using a very deep result 
of~\cite{BFI} concerning the {\it Bombieri-Vinogradov
theorem\/}. 

For an integer $k\ge 1$ we write
$$
T(k) =  \max_{i=1,\ldots,\tau(k)-1} \frac{d_{i+1}}{d_i}, 
$$
where $1= d_1<\ldots<d_{\tau(k)}=k$ are the positive divisors of $k$.

By~\cite[Theorem~1]{Saias}, we have:

\begin{lemma}
\label{lem:div dens}
Uniformly in  $z \ge t\ge 2$,
$$
\frac{z \log t}{\log z} \gg
\# \left\{k \le z \ : \  T(k) \le t \right\} \gg
\frac{z \log t}{\log z}.
$$
\end{lemma}

Finally, we remark, that several interesting results about the 
distribution of points on  $\cH_{a,m}(\cX,\cY)$ on average over $a$,
for some
special sets $\cX$ and $\cY$, such as intervals 
$$
\cX = \cY = \{z \ : \ 1 \le z \le m/2\},
$$ 
are based on various asymptotic
formulas for average values of Dirichlet $L$-functions, see, for
example,~\cite{LiuZha1,XuZha1,Zha6}. 

\subsection{Integer Points on Algebraic Surfaces and Convex Polygons}
\label{sec:Dioph}

The result of~\cite[Theorem~2]{KonShp} is based on the following 
bound on the number of solutions to a bivariate quadratic 
Diophantine equation, which follows from~\cite[Lemma~3.5]{VaWo}
combined with~\cite[Proposition~1]{CillGar} 
and~\cite[Theorem~1]{Shel}:

\begin{lemma}
\label{lem:Quadratic}
Let
$$
G(X,Y)= AX^2  + BXY + C Y^2 + DX + EY + F\in \Z[X,Y]
$$
be an irreducible quadratic polynomial with coefficients of
size at most $H$. Assume that $G(X,Y)$ is not affine
equivalent to a parabola $Y = X^2$ and has a nonzero determinant
$$
\Delta = B^2-4AC \ne 0.
$$
 Then the equation $G(x,y)=0$ has
at most $H^{o(1)}$ integral solutions $(x,y) \in [0,H]\times[0,H]$.
\end{lemma}

As usual, we say that a polygon $\cP \subseteq \R^2$ is integral if
all its vertices belong to the integral lattice $\Z^2$.

Also, following~\cite{Arn} we say two  polygons $\cP, \cQ \subseteq \R^2$ are equivalent
if there is an affine transformation
$$T: \vec{x} \mapsto A\vec{x} + \vec{b}, \qquad \vec{x} \in \R^2, 
$$
for $A = \GL_2(\Z)$ and $\vec{b} \in \Z^2$ preserving the integral
lattice $\Z^2$ (that is,  $\det A = \pm 1$) that maps $\cP$ to
$\cQ$.

The following result of~\cite[Lemma~3]{BarPach} plays an important role
in the argument of~\cite{KonShp}

\begin{lemma}
\label{lem:Equiv}   An integral polygon of area $S$ is equivalent to
a polygon contained in some box $[0,u]\times[0,v]$ of area $uv\le 4 S$.
\end{lemma}

Besides, the approach of~\cite{KonShp} also makes use a special case 
of 
the following general result of~\cite[Lemma 2.2]{Pet}
which we use only in $\R^2$. 

\begin{lemma}
\label{lem:Pet}
Let $\fU\subseteq \R^d$ be a convex compact. We consider
a finite sequence of compacts $\fV_i\subseteq K$, $i=1, \ldots, n$,
such that none of them meets the convex hull of others. Then
$$\sum_{i=1}^n (\vol \fV_i)^{(d-1)/(d+1)}\ll (\vol \fU)^{(d-1)/(d+1)},$$
where  $\vol \fA$ denotes the volume
of a compact set $\fA\subseteq \R^d$ and
 the implied constant depends only on $d$.
\end{lemma}

Furthermore, a result of ~\cite{FKS} (that has been 
improved in~\cite{KonShp}) uses the bound  $O(S^{1/3})$
of~\cite{Andrews} on the number of integer vertices of
a convex polygon of area $S$.

\section{Distribution of Points on $\cH_{a,m}$}

  \subsection{Points on $\cH_{a,m}$ in Intervals for All $a$}
\label{sec:Indiv}

A classical conjecture asserts that for  any fixed $\eps > 0$
and a sufficiently large $p$, 
for every integer $a$ there are  integers $x$ and $y$ with
$|x|,|y|\le p^{1/2 + \eps}$
such that
$xy \equiv a \pmod p$; 
see~\cite{Gar1,GarKar1,GarKar2,GarKu} and references therein.
The question has probably been motivated by the following observation.
Using the Dirichlet
pigeon-hole principle, one can easily show that
for every integer $a$ there are integers $x$ and $y$ with
$|x|,|y|\le p^{1/2}$  and $x/y \equiv a \pmod p$.

Unfortunately, this  is known only with
$|x|,|y| \le C p^{3/4 }$ for some absolute constant $C>0$, which is
shown in~\cite{Gar2}.
Several modifications of this bound, for example for composite $m$,
are also known,
see~\cite{KhShp}. These results are based on the
bound~\eqref{eq:Kloost bound} of 
Kloosterman sums~\eqref{eq:Kloost}
(and its more precise form in the case when $m=p$ is a prime) combined with some
other standard arguments. The same arguments also produce the following estimate
which is a slight generalisation of several previously known
results, see~\cite{BeKh,FuKi} and references therein. 
This estimate is certainly very well-known and has 
 appeared in the literature in various forms. We however give a 
short proof in order to demonstrate the underlying techniques.

\begin{theorem}
\label{thm:UD Gen}  Let $\cX = \{U+1, \ldots, U+X\}$, where $m> X \ge 1$ and
$U \ge 0$ are arbitrary integers.
  Suppose that for every $x \in \cX$ we are given a set
$\cY_x =
\{V_x+1,
\ldots, V_x+Y\}$
where $m> Y \ge 1$ and $V_x \ge 0$ are arbitrary integers.
Then for any integer $m\ge 1$ and $a$ with $\gcd(a,m)=1$, we have 
$$
\sum_{\substack{(x,y) \in
\cH_{a,m} \\ x \in \cX, y \in \cY_x}} 1 =   \frac{\varphi(m)}{m^2} XY 
+ O(m^{1/2 +  o(1)}).
$$
\end{theorem}

\begin{proof} Using~\eqref{eq:Ident add char} we write
\begin{eqnarray*}
\sum_{\substack{(x,y) \in
\cH_{a,m} \\ x \in \cX,  y \in \cY_x}} 1
  & = &\frac{1}{m^2} \sum_{\substack{(x,y) \in
\cH_{a,m} \\ 1 \le x,y \le m}}
\sum_{w \in \cX} \sum_{z \in \cY_w} \sum_{r,s\in \Zm}\em\(r(x-w) +
s(y-z)\) \\ & = &\frac{1}{m^2} \sum_{r,s\in \Zm}K_m(r,as)
\sum_{w \in \cX} \em\(-rw\) \sum_{z \in \cY_w}  \em\(-sz\).
\end{eqnarray*}
We now separate the main term which corresponds to $r=s = 0$
and  is equal to
$$
\frac{XY}{m^2} \sum_{\substack{(x,y) \in \cH_{a,m} \\ 1 \le x,y \le m}}   1
=  \frac{\varphi(m)}{m^2} XY.
$$
For the error term $\Err$,  for each divisor $d|m$, we collect
together pairs $(r,s)$ with
the same value $\gcd(r,s,m) = d$.

Applying the bounds~\eqref{eq:Kloost bound}  and~\eqref{eq:Incompl} we obtain
\begin{eqnarray*}
|\Err| & \le &  m^{1/2 + o(1)}   \sum_{\substack{d|m\\ d < m}} d^{1/2}
\sum_{\substack{-(m-1)/2 \le r,s\le
m/2\\\gcd(r,s,m) = d}} \frac{1}{(|r|+1)(|s|+1)}\\
& \le &   m^{1/2 + o(1)}   \sum_{\substack{d|m\\ d < m}} d^{1/2}
\(\sum_{|t| \le
m/2d} \frac{1}{d  |t| + 1}\)^2\\
& \le &   m^{1/2 + o(1)}   \sum_{\substack{d|m\\ d < m}} d^{-3/2}
\le  m^{1/2 + o(1)}  ,
\end{eqnarray*}
which leads to the
desired statement.
\end{proof}

For example if $V_x = V$ for all $x \in \cX = \{U+1, \ldots, U+X\}$
and  $\cY =\{V+1,  \ldots, V+Y\}$ then Theorem~\ref{thm:UD Gen}
yields
\begin{equation}
\label{eq:Standard}
   \# \cH_{a,m}(\cX,\cY) =   \frac{\varphi(m)}{m^2} XY 
+ O(m^{1/2 +  o(1)}).
\end{equation}

It seems that improving  Theorem~\ref{thm:UD Gen} or even
just the
asymptotic formula~\eqref{eq:Standard} and making them  nontrivial for $XY <
m^{\alpha}$ with some fixed
$\alpha < 3/2$   is out of reach at the present time. 
In particular,  this exponent is related to the fact that 
an asymptotic formula for the sum of $\tau(n)$ for integers $n \le   X$
in an arithmetic progression $n \equiv a \pmod m$ is known only 
for $m \le X^{2/3-\varepsilon}$ for an arbitrary fixed $\varepsilon > 0$.
This result has been independently discovered by  A.~Selberg and C.~Hooley 
(see, for example, the discussion in~\cite{Hool1}) and has resisted any improvement
for more than half a century.  

In turn, any improvement of Theorem~\ref{thm:UD Gen} 
or just of~\eqref{eq:Standard} will trigger 
a chain reaction of improvements in many other problems; in particular some of them
are outlined in this survey. Probably the most feasible way to tackle this
problem is to obtain good bounds of incomplete Kloosterman sums,
improving the estimate 
$$
\left| \sum_{\substack{ V+1 \le x \le V+X\\ \gcd(x,m)=1}}
\em\(sx^{-1}\) \right|  \le \(m \gcd(s,m)\)^{1/2+o(1)},
\qquad 1 \le X \le m,
$$
which is immediate from~\eqref{eq:Kloost bound}. 
We recall that  a famous conjecture of 
Hooley~\cite{Hool2} asserts that if $\gcd(s,m)=1$ then the bound
\begin{equation}
\label{eq:HooleyConj}
\left| \sum_{\substack{1 \le x \le X\\ \gcd(x,m)=1}}
\em\( sx^{-1}\) \right|
\le X^{1/2}m^{o(1)},
\end{equation}
holds uniformly over $m^{1/4} \le X \le m$  
The conjectured bound~\eqref{eq:HooleyConj} enables us to
derive that  for $  \cX = \{1, \ldots, X\}$ with $X \ge m^{1/4}$ 
and  $\cY =\{V+1,  \ldots, V+Y\}$ 
we have    
$$\# \cH_{a,m}(\cX,\cY) =   \frac{\varphi(m)}{m^2} XY 
+ O(X^{1/2}m^{o(1)}).
$$
For example, we see that assuming~\eqref{eq:HooleyConj}  
one can show that for any $\eps$ and sufficiently large prime $p$,
for every integer $a$ there are  integers $x$ and $y$ with
$|x|,|y|\le p^{2/3 + \eps}$
and such that $xy \equiv a \pmod p$. 
Unfortunately, the conjecture~\eqref{eq:HooleyConj}   seems
to be extremely difficult; see however~\cite{Kor} for some 
improvements of the bound~\eqref{eq:Kloost bound}.

On the other hand, there are
some apparently easier questions which could be more feasible to answer.

\begin{question} Improve the
asymptotic formula~\eqref{eq:Standard}  for some
special moduli $m$,
such as primes or prime powers.
\end{question}

It is quite possible that the method and results of~\cite{Fouv1}
can be use to give a positive answer to the next problem; see~\cite{FKS}
where the results of~\cite{Fouv1} have been used to study 
the distribution of points on $\cH_{a,m}$.

\begin{question} Improve the
asymptotic formula~\eqref{eq:Standard} on average over the
moduli $m \le M$.
\end{question}

Clearly, Theorem~\ref{thm:UD Gen}  can be viewed as the bound
$O(m^{-1/2 + o(1)})$
on the discrepancy of the points~\eqref{eq:Map}. Thus one can now
apply Lemmas~\ref{lem:LNW bound}  and~\eqref{lem:Weyl bound} to study the
distribution of points on $\cH_{a,m}$ in more complicated domains than
boxes $[U+1, U+X]\times [V+1,V+Y]$ covered by Theorem~\ref{thm:UD Gen}.
For example, we immediately deduce that
$$
\# \{(x,y) \in \Hba\ : \ x^2 + y^2 \le r^2\} = \frac{\pi r^2 
\varphi(m) }{4 m^2}
+ O(m^{3/4 + o(1)}).
$$

Furthermore,
the asymptotic formula~\eqref{eq:Standard}, combined with
Lemma~\ref{lem:K-H}, provides the most direct way to the following
asymptotic formula
\begin{equation}
\label{eq:Dist Moments}
\sum_{(x,y) \in \Hba} (x-y)^{2\nu}
= \frac{1}{(2\nu + 1)(\nu + 1)}m^{2\nu} \varphi(m) + O\(m^{2\nu + 1/2 +
o(1)}\)
\end{equation}
which has been given in~\cite{Zha4} (in a slightly more
precise form, which can also be obtained within our
elementary arguments).

Similarly Theorem~\ref{thm:UD Gen}  and
Lemma~\ref{lem:K-H} imply that for any
real positive $\Delta< m/2$
\begin{equation}
\label{eq:Dist Distr}
\sum_{\substack{(x,y) \in \Hba  \\ |x-y| \le
\Delta}} 1  = \frac{\Delta(2m-\Delta)\varphi(m)}{m^2} +
O\(m^{ 1/2 + o(1)}\)
\end{equation}
which is a version of a result of~\cite{Zha5}.

  As we have mentioned, Theorem~\ref{thm:UD Gen}
uses very little specific information about the
congruence $xy \equiv a \pmod m$ and  can be extended to many
other congruences. For  prime
$m=p$ one can use the Bombieri bound~\cite{Bomb}
   instead of~\eqref{eq:Kloost bound}
and obtain exactly the same results in much
more general settings.
For example, this has been done for solutions of polynomial  congruences
modulo $p$ and also for joint distribution of inverses
modulo $p$ of $s$ linear  forms $a_j x + b_j$, $j =1, \ldots, s$, with
integer  coefficients,
see~\cite{Chan1,CoGoZa,CoVaZa2,CoZa1,CoZa2,CoZa3,GrShZa,
VajZah,Zah,Zha5.6,ZhaYi,Zhe} and references therein.

We also remark a very nice and completely elementary result 
of~\cite{Chan4} about the distances between inverses on 
pairs of consecutive integers.

It is also easy to use some other  bounds of more general exponential 
sums (instead of~\eqref{eq:Kloost bound}) to study the distribution 
of values of polynomials and rational functions on the points of $\cH_{a,m}$. 

For example, in~\cite{YiZha,Zha5.6} some questions are studied about the distribution
of residues modulo $m$ of
powers $(x^k, y^k)$ taken over all $(x,y)\in \Hba$.
In particular, it has been shown in~\cite{YiZha},
that for any fixed integer $k\ne 0$, the smallest positive residue 
modulo $m$ of  $n^k$ and 
its modular inverse $\inv{n^k}$ are of the same parity
\begin{equation}
\label{eq:Weak Nkm}
N(k,m) = 0.5 \varphi(m) + O\(m^{3/4 + o(1)}\)
\end{equation}
times when $n$ runs through all invertible elements
of  $\Z/m\Z$.

These results can be substantially extended and
improved if one uses Lemma~\ref{lem:Exp Sparse}.
In particular, it has been shown in~\cite{Shp9} that 
Lemma~\ref{lem:Exp Sparse} immediately implies that the error
term can be
lowered from $m^{3/4+o(1)}$ to  $m^{1/2+o(1)}$ 
in~\eqref{eq:Weak Nkm}, that is,  we have
$$
N(k,m) = 0.5 \varphi(m) + O\(m^{1/2 + o(1)}\). 
$$
Moreover, it is shown in~\cite{Shp9} that an analogous asymptotic 
formula (with the error term $m^{1 - 1/s+o(1)}$) for the 
counting function of the
number of residues of $s$ distinct powers $x^{k_1}, \ldots, x^{k_s}$,
where $k_1, \ldots, k_s \in \Z\backslash\{0\}$, 
of $x$ modulo $m$ which belong to $s$ prescribed arithmetic progressions. 
All these results can be obtained by combining
Lemma~\ref{lem:Exp Sparse}
with standard arguments similar to those used in the proof of
Theorem~\ref{thm:UD Gen}.

In~\cite{Byk1,Ust4} one can find  asymptotic formulas
for  the number  of points of $(x,y) \in \cH_{a,m}$
in more complicated regions of the form
\begin{equation}
\label{eq:Curve}
x \in  \{U+1, \ldots, U+X\}, \qquad 0 \le y  \le f(x),
\end{equation}
where $f(t)$ is a twice differentiable function which satisfies 
$$
F \ll |f''(t)| \ll F, \qquad  t \in [U+1,U+X],
$$
for some $F\ge 1$ (the error term also involves $F$). 
It is shown in~\cite{Byk1} that this question has applications
to sums of divisor function with quadratic polynomials. Further 
investigations in this directions would be of great interest.

\subsection{Points on $\cH_{a,m}$ in Intervals on Average Over $a$}
\label{sec:Aver}

It is natural to expect that one can get stronger results than
Theorem~\ref{thm:UD Gen} on average over $a$.

This indeed is true, and it has been shown in the series of
   works~\cite{Gar1,GarKar1,GarKar2,GarKu}
that the congruence $a \equiv xy \pmod m$ is solvable for all but $o(m)$ values
of  $a = 1, \ldots, m-1$,  with $x$ and $y$
significantly smaller than $m^{3/4}$.
In particular,  in~\cite{GarKar2}, this is proved 
for  $x$ and $y$  in the range $1\le x,y\le m^{1/2}(\log m)^{1 + \eps}$.
Certainly this result is very sharp.
Indeed, it has been noticed in~\cite{Gar1} that   well-known estimates for
integers with a  divisor
in a given interval (see~\cite{Ford,HalTen}) immediately imply that for any
$\varepsilon > 0$  almost all residue
classes modulo $m$ are not of the form $xy \pmod m$ with
$1 \le x,y \le m^{1/2}(\log m)^{\kappa -\varepsilon}$ where $\kappa$ is
given by~\eqref{eq:kappa}.

One can also derive from~\cite{FKSY} (which in turn
makes use of Lemma~\ref{lem:div int}) that for any
$\varepsilon>0$ the  inequality
$$
\max\{|x|,|y| \ : \ xy \equiv 1 \pmod m\} \ge m^{1/2} (\log
m)^{\kappa/2} (\log\log m)^{3/4 - \varepsilon}
$$
holds:
\begin{itemize}
\item for  all positive integers $m \le M$, except for possibly $o(M)$ of
them,

\item for all prime $m = p \le M$ except for possibly  $o(M/\log M)$ of
them.
\end{itemize}

In~\cite{FKS}, some ideas and results of~\cite{Fouv1} have been 
used to show that 
$$
\max\{|x|,|y| \ : \ xy \equiv 1 \pmod m\} \le m^{1/2+o(1)} 
$$
for  all positive integers $m \le M$, except for possibly $o(M)$ of
them. So this implies that the above estimates of~\cite{FKSY}  are quite tights and also settles a slightly relaxed form 
of one of the conjectures of~\cite{FKSY}.

Similar questions about the ratios $x/y$,
have also been studied,  see~\cite{Gar1,GarKar2,Sem1}.

The result of~\cite{GarKar2} shows that almost all reduced classes modulo $m$
can be represented  as  $xy$ with $1\le
x,y\le m^{1/2 +\eps}$. However, it does not imply that these products
are uniformly
distributed in reduced residue classes, which is sometimes required in
applications.

In this respect, the following bound is a minor modification of
a result of~\cite{Shp6} and gives the desired uniformity of
distribution  for $1 \le x \le X$, $V+1 \le y \le V+Y$
provided that $X,  Y\ge m^{1/2 + \varepsilon}$
for a fixed $\varepsilon > 0$ and sufficiently large integer $m$.
In turn, it is based on some ideas from~\cite{BHS}.

\begin{theorem}
\label{thm:UD Aver}  Let $\cX = \{1, \ldots, X\}$ and $\cY = \{V+1,
\ldots, V+Y\}$
where $X,Y \ge 1$ and $V \ge 0$ are arbitrary integers.
Then for any integer $m\ge 1$,
$$
\sum_{\substack{a=1 \\ \gcd(a,m) =1}}^m \left|\# \cH_{a,m}(\cX,\cY) -
 \frac{\varphi(m)}{m^2} XY \right|^2
\le X(X+Y)  m^{o(1)}.
$$
\end{theorem}

\begin{proof}   We rewrite the congruence $xy \equiv a \pmod m$ as
   $y \equiv ax^{-1} \pmod m$ (where the inversion is taken modulo $m$).
Using the identity~\eqref{eq:Ident add char}, we write
\begin{eqnarray*}
\lefteqn{\# \cH_{a,m}(\cX,\cY) = \frac{1}{m} \sum_{\substack{x =1\\
\gcd(x,m)=1}}^X \sum_{y =V+1}^{V+Y}
\sum_{-(m-1)/2 \le r \le m/2 }\em\(r(ax^{-1}-y)\) }\\
& & \qquad =
\frac{1}{m}\sum_{-(m-1)/2 \le r \le m/2 }
\em(-rV)\sum_{\substack{x=1\\ \gcd(x,m)=1}}^X
\em\(arx^{-1}\)\sum_{y=1}^Y \em(-ry).
\end{eqnarray*}

By Lemma~\ref{lem:erat}, the main term corresponding to $r=0$ is
$$\frac{1}{m} \sum_{\substack{x=1\\ \gcd(x,m)=1}}^X
   \sum_{y =1}^Y 1
=    \frac{\varphi(m)}{m^2}   XY  + O\(Ym^{-1+ o(1)}\).
$$
Hence
$$
\# \cH_{a,m}(\cX,\cY)-  \frac{\varphi(m)}{m^2}   XY   \ll \frac{1}{m}
|\Err_{a,m}(X,Y)| +  Ym^{-1 + o(1)},
$$
where
$$\Err_{a,m}(X,Y) = \sum_{1 \le |r| \le m/2 }  
\sum_{\substack{x=1\\ \gcd(x,m)=1}}^X\em\(arx^{-1}\)\sum_{y =1}^Y \em(-ry).
$$
Using the Cauchy inequality (and the extending the summation 
to all residue classes modulo $m$), we derive
\begin{equation}
\label{eq:H and E}
\begin{split}
\sum_{\substack{a=1 \\ \gcd(a,m) =1}}^m \left|\# \cH_{a,m}(\cX,\cY) - \frac{\varphi(m)}{m^2} XY 
\right|^2&\\
\ll \frac{1}{m^2} \sum_{a =1}^m &|\Err_{a,m}(X,Y)|^2 +  Y^2m^{-1 + o(1)}.
\end{split}
\end{equation}

We now put $J=\fl{\log (Y/2) }$ and define the sets
\begin{eqnarray*}
\cR_0 &=& \left\{r\ : \ 1\le|r|\le \frac{m}{Y} \right\},\\
\cR_j&=&\left\{r\ : \ e^{j-1}\frac{m}{Y} <|r|\le e^{j}\frac{m}{Y}\right\},
\qquad j =1, \ldots, J,\\
\cR_{J+1}&=&\left\{r\ : \ e^{J}\frac{m}{Y} <|r|\le m/2 \right\}
\end{eqnarray*}
(we can certainly assume that $J \ge 1$ since otherwise the
bound is trivial).

Applying the Cauchy inequality again, we deduce
\begin{equation}
\label{eq: E and Ej}
|\Err_{a,m}(X,Y)|^2  \le (J+2) \sum_{j=0}^{J+1} |\Err_{a,m,j}(X,Y)|^2,
\end{equation}
where
$$
   \Err_{a,m,j}(X,Y) = \sum_{r \in \cR_j}   \sum_{\substack{x=1\\
\gcd(x,m)=1}}^X
\em\(arx^{-1}\) \sum_{y
=1}^Y \em(-ry).
$$
Using~\eqref{eq:Incompl}, we conclude that
$$
\sum_{y =1}^Y \em(-ry)  \ll e^{-j} Y.
$$
for $r \in \cR_j$,  $j  = 0, \ldots, J+1$.
Thus
$$
\Err_{a,m,j}(X,Y) \ll e^{-j} Y \left| \sum_{r \in \cR_j}  \vartheta_r
\sum_{\substack{x=1\\
\gcd(x,m)=1}}^X\em\(arx^{-1}\) \right| ,  \qquad j  = 0, \ldots, J+1,
$$
for some complex numbers $ \vartheta_r$ with $| \vartheta_r | \le 1$
for $|r| \le m/2$.
Therefore,
\begin{eqnarray*}
\lefteqn{
\sum_{a =1}^m |\Err_{a,m,j}(X,Y)|^2 \ll  e^{-2j} Y^2
   \sum_{a=1}^m \left| \sum_{r \in \cR_j}  \vartheta_r
\sum_{\substack{x=1\\ \gcd(x,m)=1}}^X
\em\(arx^{-1}\) \right|^2
}\\
&  & \qquad =   e^{-2j} Y^2 \sum_{r_1,r_2 \in \cR_j}
\vartheta_{r_1}\vartheta_{r_2}
\sum_{\substack{x_1, x_2 \le X\\
\gcd(x_1x_2,m)=1}}^X  \sum_{a=1}^m  \em\(a \(r_1x_1^{-1}- r_2x_2^{-1}\)\).
\end{eqnarray*}
Clearly the inner sum vanishes if $r_1x_1^{-1} \not \equiv
r_2x_2^{-1} \pmod m$
and is equal to $m$ otherwise. Therefore
\begin{equation}
\label{eq: sum E_j}
\sum_{a =1}^m |\Err_{a,m,j}(X,Y)|^2 \ll e^{-2j} Y^2   m T_j ,
\end{equation}
where $T_j$ is the number of solutions to the congruence
$$
r_1x_2 \equiv  r_2x_1 \pmod m, \qquad
r_1,r_2 \in \cR_j, \ x_1, x_2  \le X, \
\gcd(x_1x_2,m)=1.
$$
We now see that if $r_1$ and $x_2$ are fixed,  then $r_2$ and $x_1$ are such
that their product $s = r_2 x_1 \ll e^j mX/Y$ belongs to a
prescribed residue class modulo $m$. Thus there are at most $O\(e^jX/Y + 1\)$
possible values of $s$  and
for each fixed $s \ll e^j mX/Y$ there are $\tau(s) = m^{o(1)}$ values of
   $r_1$ and $x_2$ with $s = r_1 x_2$, see~\eqref{eq:tau}.
Therefore
$$
T_j \le X \# \cR_j \(e^jX/Y + 1\) m^{o(1)} =
\frac{e^{2j}X^2m^{1+o(1)}}{Y^2}+  \frac{e^{j}Xm^{1+o(1)}}{Y}
$$
and after substitution into~\eqref{eq: sum E_j}
we get
$$
   \sum_{a =1}^m |\Err_{a,m,j}(X,Y)|^2 \ll e^{-2j} Y^2   m T_j =
X^2m^{2+o(1)} + e^{-j}XY m^{2+o(1)} .
$$
Substituting  this bound in~\eqref{eq: E and Ej}  and
recalling~\eqref{eq:H and E}, we conclude the proof.
\end{proof}

We note that the proof of Theorem~\ref{thm:UD Aver}
can easily be extended to
arbitrary sets  $\cX \subseteq \{1, \ldots, X\}$, see~\cite{Shp6}.
However, it  breaks down
if $x$ runs through a short interval away from the origin.
An alternative approach has been suggested in~\cite{GarGar} and is
based on bounds of the fourth moment of multiplicative 
character sums,
see  Lemma~\ref{lem:4th Moment}.
If $m=p$ is prime, it 
can handle such shifted intervals $\cX = \{U+1, \ldots, U+X\}$
(but  not arbitrary sets 
$\cX \subseteq \{1, \ldots, X\}$ as that of~\cite{Shp6}).  
Furthermore, the technique of~\cite{GarGar} leads to more explicit 
expressions instead of $m^{o(1)}$ in the error term. 
Thus, although the approaches of~\cite{GarGar} and~\cite{Shp6} complement
each other they still leave some natural open questions.

\begin{question}
\label{quest:Shift Int} Extend Theorem~\ref{thm:UD Aver} to sets
$\cX = \{U+1, \ldots, U+X\}$ with arbitrary $U$.
\end{question}

As we have mentioned, if $m=p$, Question~\ref{quest:Shift Int}
is addressed in~\cite{GarGar}, however  some of the necessary
ingredients are not known for composite $m$. 

In~\cite{Chan5} the behaviour of $\# \cH_{a,m}(\cX,\cY)$ 
has been  studied for the sets 
 $\cX = \{U+1, \ldots, U+X\}$ and $\cY = \{V+1, \ldots, V+Y\}$
 on average over $U$ and $V$. It is shown in~\cite{Chan5} that
$\# \cH_{a,m}(\cX,\cY)$  is close to its expected value $XY \varphi(m)/m^2$
for almost all $U,V \in \Z_m$ provided that 
$X,Y \ge m^{1/2 + \varepsilon}$ for some fixed 
$ \varepsilon > 0$.   Furthermore, one can also find 
in~\cite{Chan5} a similar statement for a multidimensional
analogues of   $\# \cH_{a,m}(\cX,\cY)$. Several more 
results of this kind  have been given in~\cite{GoKrSo}, 
however in the setting of~\cite{GoKrSo}  the initial 
point of only one interval is
``sliding'', while the other one is fixed, see also~\cite{BrHa}.

We also note that several results ``on average'' related
to various modifications
of the {\it Lehmer problem\/} are given 
in~\cite{LiuZha1,Wei,Xu,XuZha1,XuZha2,YuYi,ZhaXue,Zha6,ZXY}, 
see also Section~\ref{sec:Lehm}.

Finally, we remark that $\# \cH_{a,m}(\cX,\cY)$
has been studied in~\cite{CoVaZa1} for the same sets
as in Theorem~\ref{thm:UD Gen}, that is, for
$\cX = \{1, \ldots, X\}$ and $\cY = \{V+1,  \ldots, V+Y\}$,
but on average over $V$.  It is shown in~\cite{CoVaZa1}
that in this case one can also
obtain stronger bounds than that of Theorem~\ref{thm:UD Gen}.

\subsection{Points on $\cH_{a,m}$ in Sets with Arithmetic Conditions}
\label{sec:Arith Set}

Theorems~\ref{thm:UD Gen} and~\ref{thm:UD Aver} consider the case
when $x$ and $y$ belong to sets of consecutive integers. However, 
studying points on  $\cH_{a,m}$ in other sets is of ultimate
interest as well.

We start with a very simple observation that no general result
of the type of Theorems~\ref{thm:UD Gen} and~\ref{thm:UD Aver}
applying to arbitrary sets $\cX$ and $\cY$ is possible (even
for very massive sets  $\cX$ and $\cY$). For example, if $m = p$
is a prime and $\cX=\cY$ consist of all $(p-1)/2$ quadratic
residues modulo $p$, then $\cH_{a,p}(\cX,\cY) = \varnothing$
for every quadratic nonresidue $a$.

The problem of distribution of pairs of primes $(p,q) \in \Hba$
has been considered in~\cite{EOS}. Unfortunately, it seems
that even the {\it Extended Riemann Hypothesis\/} is not
powerful enough to get a satisfactory answer to this
question, see~\cite{EOS} for details.
However, on average over both $a$ and $m$ this problem
becomes more feasible and has actually been considered 
in~\cite{FrKuShp}. More precisely, let
$P(X;m,a)$ be
the number of solutions to the congruence
$$
p_1 p_2 \equiv a \pmod m
$$
in primes $p_1, p_2\le X$. It is shown in~\cite{FrKuShp} that 
for any sufficiently large $M$,
$$
\sum_{M<m \le 2M}\sum_{\substack{a=1\\ \gcd(a,m)=1}}^m \(P(X;m,a) -
\frac{\pi(X)^2}{\varphi(m)}\)^2 \ll X^4(\log X)^{-A}+ MX^2 ,
$$
for any $A$, with an implied constant that depends on $A$,
and
$$
\sum_{M<p \le 2M}\sum_{ a=1}^{p-1} \(P(X;p,a) -
\frac{\pi(X)^2}{p-1}\)^2 \ll \(M^{-1}X^4 + MX^2\)(\log X)^{-2}\ .
$$
The number $S(X;m,a)$  of solutions to the congruence
$$
s_1 s_2 \equiv a \pmod m
$$
in squarefree positive integers $s_1, s_2\le X$ is certainly 
easier to study and in this case ``individual''
results are possible. For example, it is shown in~\cite{FrKuShp} 
that for all integers $m\ge 1$ and $a$ with $\gcd(a,m) =1$
and real positive $x$, we have 
\begin{equation}
\label{eq:SF prod}
S(X;m,a) = \frac{36}{\pi^4} \cdot \frac{X^2}{m} \prod_{p\mid m} \(1 +
\frac{1}{p}-\frac{1}{p^2}+ \frac{1}{p^3}\)^{-1} + O\(X m^{-1/4 + o(1)}\). 
\end{equation}

Moreover,    representations of residue 
classes by products of a squarefree number and a
prime  have also been studied~\cite{FrKuShp}
(for a fixed modulus but on average over residue classes). 

The result of~\cite{FrShp} on the distribution of values of the
Euler function in residues classes is based on studying congruences 
with  products
of shifted primes  
$$
(p_1-1)(p_2-1)(p_3-1) \equiv a \pmod  m
$$ 
and employing  bounds 
of multiplicative character sums with shifted primes 
from~\cite{Kar0,Rakh}; see also~\cite{FrLuc} for some related 
results on some  residue classes modulo $m$ which are ``hard''
to represent by a totient.

In~\cite{Gar3} an improvement of some results of~\cite{FrShp} has been
obtained. This new idea of~\cite{Gar3} is to use a
bound of~\cite{Hux} on the number of large values of 
Dirichlet polynomials (see also~\cite[Chapter~9]{IwKow}),
and can probably be applied to a
number of other questions. For example, one of such
applications to the distribution of points on  multivariate 
modular hyperbolas is given in~\cite{Gar6}. 

It is shown in~\cite{FrKuShp} that there are two 
absolute constants $\eta,\kappa>0$
such that for any prime $p$ and integer $X$ 
with $p > X \ge p^{1-\eta}$, if $\gcd(ak,p)=1$ then the congruence
$$
p_1p_2(p_3+k) \equiv a \pmod p
$$ 
has $(1 + O(p^{-\kappa}))\pi(X)^3/p$ solutions in 
primes $p_1, p_2, p_3 \leq X$ (uniformly over $a$ and $k$).  
It is shown in~\cite{Gar6}, using a result 
of~\cite{Hux}, that one can take any $\eta < 13/76$. 

In~\cite{Gar5}, for $X \le p$,  the bound
\begin{equation}
\label{eq:recipr prime inviv}
\max_{\gcd(b,p)=1} \left|\sum_{\substack{q \le X\\q~\mathrm{prime}}} \ep(bq^{-1})\right|
\ll \(X^{15/16} + X^{2/3}p^{1/4}\)p^{o(1)}
\end{equation}
has been obtained
for exponential sums of reciprocals of primes, which is
in this special case   an improvement of a more general
result of~\cite{FoMi}. In turn, this
has led to an improvement of the result of~\cite{FrKuShp}
for the number of solutions of the congruence
\begin{equation}
\label{eq:3 prime cong}
p_1(p_2+p_3) \equiv a \pmod p
\end{equation}
in primes $p_1, p_2, p_3 \leq X$. More precisely, in~\cite{Gar5}, 
the asymptotic formula 
$(1 + O(p^{-\kappa}))\pi(X)^3/p$ for the number of solutions
is shown to hold for $p\ge X  \ge p^{16/17 + \varepsilon}$, 
where $\kappa > 0$ depends only on $\varepsilon>0$. Note that the exponent 
$16/17$ improves the previous exponent $38/39$ of~\cite{FrKuShp} that is
based on the estimate of~\cite{FoMi}. Using the bound of~\cite{FoShp1}
on  the average values
of such  exponential sums
\begin{equation}
\label{eq:recipr prime aver}
\sum_{Z \le p \le 2Z} \max_{\gcd(b,p)=1} \left|\sum_{\substack{q \le X\\q~\mathrm{prime}}} \ep(bq^{-1})\right|
\ll \(X^{3/5}Z^{13/10} + X^{5/6}Z^{13/12}\)Z^{o(1)}, 
\end{equation}
where $X \le Z$,
for almost all primes $p$, one can obtain a similar result 
for $p\ge X  \ge p^{13/14 + \varepsilon}$. 
We note that the both bounds~\eqref{eq:recipr prime inviv} 
and~\eqref{eq:recipr prime aver} are nontrivial for 
$X \ge p^{3/4+\varepsilon}$ for some fixed $\varepsilon>0$.
The bound of~\cite[Theorem~4.1]{Bour1} is nontrivial for 
$X \ge p^{1/2+\varepsilon}$, however it is less explicit.
Recently, an explicit 
version of the results of~\cite{Bour1} has been given in~\cite{Bak}
how it is not immediately clear whether it can lead to new results about the 
congruence~\eqref{eq:3 prime cong} (except for the case where
the variables are in the domain $p_1\le X$, $p_2,p_3 \le Y$
with $X$ substantially smaller than $Y$). 

Furthermore, yet another approximation 
to the initial problem has been considered in~\cite{Shp15}.
Combining  a generalisation of the bound~\eqref{eq:recipr prime inviv} 
to composite denominators (given in~\cite{FoShp1}) with the sieve
method, it is shown in~\cite{Shp15} that for a sufficiently 
large $m$ and any $a$ with $\gcd(a,m)=1$
there is a solution to the congruence 
$$
pP_{17} \equiv a \pmod m, \qquad p,P_{17} \le m^{0.997}, 
$$
where $p$ is prime and $P_{17}$ has at most 
$17$ prime divisors.

One can also study the distribution of
points $(x,y) \in  \cH_{a,m}(\cX,\cY)$ with some prescribed
structure of prime factors. For example, let $P_+(k)$ and
$P_-(k)$ denote the largest and the smallest prime divisors
of an integer $k \ge 1$, respectively. 

\begin{question}
\label{quest:P+/P_ sols}
Obtain an asymptotic formula for
$$
\# \{(x,y) \in  \cH_{a,m}(\cX,\cY) \ :\  P_+(xy) \le R\}
$$
and
$$
\# \{(x,y) \in  \cH_{a,m}(\cX,\cY) \ :\  P_-(xy) \ge r\}
$$
where  $\cX = \{1, \ldots, X\}$ and $\cY = \{1,
\ldots, Y\}$ and  $ 1 \le X,Y \le m$ are arbitrary integers
with $R$ and $r$ in reasonably large ranges.
\end{question}

We remark that   using  elementary sieving arguments
one can extend the result of Theorem~\ref{thm:UD Gen}
to counting  $(x,y) \in \cH_{a,m}$ such that $x$ is
of the largest possible multiplicative order modulo $m$
(and thus so is $y$)
which is given by the {\it Carmichael function\/} $\lambda(n)$. 
In particular, when $m=p$ is a prime, this addresses the problem of
distribution of the points 
$(x,y) \in \cH_{a,m}$ where $x$ is a primitive
root modulo $p$, for example, see~\cite{BeKh,KhShp}.
Proofs of these results usually follow the same
standard lines as the proof
of Theorem~\ref{thm:UD Gen}, except that instead
of~\eqref{eq:Kloost bound} one uses the bound of the same strength
on Kloosterman sums~\eqref{eq:Kloost} 
twisted with multiplicative characters
$$
 \sum_{\substack{(x,y) \in \cH_{m} \\ 1 \le x,y \le m}}
\chi(x) \em\(rx + sy\) 
\ll \(m \gcd(r,s,m)\)^{1/2+o(1)}.
$$
There are still some delicate issues of getting the
$m^{o(1)}$ term as small as possible.   
A certain modification of this question has 
been studied in~\cite{ZheCoch}. 

We also note that several very interesting results
have recently been obtained in~\cite{Mosh1}
about points  $(x,y) \in \Hba$ such that $x$ and $y$ have
restricted $g$-ary expansions to some fixed base $g\ge 2$, 
see also~\cite{MoshShkr}. 

Finally, we note that the dual question about the arithmetic structure 
of the ratio $(xy - a)/m$ for $(x,y) \in \cH_{a,m}$ is also
of interest and may have various applications.
For example, in~\cite{Stew} such ratios without 
small prime divisors have been studied, which in 
turn has been important for yet another number theoretic question.

\section{Geometric Properties of $\cH_{a,m}$}

\subsection{Distances}
\label{sec:Dist}

We observe that the asymptotic formulas~\eqref{eq:Dist Moments}
and~\eqref{eq:Dist Distr} have a natural interpretation as the
bounds on the power moments and the distribution function
of the distances between an
element $x \in \{1, \ldots, m\}$ with $\gcd(x,m)=1$
and its modular inverse. Several results about the
average (over $a$) value  of power moments of distances
can be found in~\cite{LiuZha1,LiuZha2,LiuZha3,LiuZha4,LiuZha5,LiuZha6}, 
see also references therein.

We now define the {\it width\/} $w_{a,m}$ of the set $\Hba$:
$$
w_{a,m} =
\max \left\{  |x-y|  \ : \
(x,y) \in \Hba \right\}.
$$
We also put
$$
w_{m} = w_{1,m}
$$
for the  width  of $\Hb1$,
which has been the main object of study of~\cite{FKSY,Khan1,KhShp}.

Using Theorem~\ref{thm:UD Gen} one easily derives that
\begin{equation}
\label{eq:Large Dist}
w_{a,m} = m + O(m^{3/4 + o(1)}).
\end{equation}
Using the same arguments as in~\cite{KhShp}, one can
obtain a more precise expression for the factor $m^{o(1)}$.

On the other hand, it has been noticed in~\cite{Khan1} that
$$
m-w_{m}\geq \rf{2\sqrt{m-1}}
$$
with equality for all $m$ of the
form
\begin{equation}
\label{eq:Spec form}
m = k^2+\ell k+1
\end{equation}
with integers $k$ and $\ell$ such that
$k>0$, $0\leq \ell<2\sqrt{k}+1$ and hence
\begin{equation}
\label{eq:liminf}
\liminf_{m\to\infty}\frac{m-w_{m}}{\sqrt{m}}=2.
\end{equation}

\begin{question}
\label{quest:Special prime}
Show that there are infinitely many primes $m = p$
of the form~\eqref{eq:Spec form} with $0\leq \ell<2\sqrt{k}+1$.
\end{question}

As a curiosity, we recall the following,  it has  been noted in~\cite{FKSY},
that $m-w_{m} \le \sqrt{8m}$ for all positive integers
$m = 2^s$ with $s\in \Z$.
Indeed, if $s$ is even,  then $m = (2^{s/2}-1)^2 + 2( 2^{s/2}-1) + 1$ 
is of the form~\eqref{eq:Spec form}.
If $s$ is odd,  then  this follows from
$$
(2^{(s+1)/2}-1)(2^s-2^{(s+1)/2}-1) \equiv 1 \pmod {2^s}.
$$
In the opposite direction it is shown in~\cite{FKSY} that
\begin{equation}
\label{eq:limsup}
\limsup_{m\to\infty}\frac{m-w_{m}}{\sqrt{m}}=\infty.
\end{equation}
Furthermore, analogues of~\eqref{eq:liminf} and~\eqref{eq:limsup}
also hold for prime values $m=p$:
$$
\liminf_{p\to\infty}\frac{p-w_{p}}{\sqrt{p}}=2\mand
\limsup_{p\to\infty}\frac{p-w_{p}}{\sqrt{p}}=\infty,
$$
which follow from the following two results given in~\cite{FKSY}.

\begin{theorem}
For infinitely many primes $p$, we have
$$
p-w_{p} \le 2\sqrt{p} + \frac{\sqrt{p}}{\log p}.
$$
\end{theorem}

\begin{proof} Let $\eps=1/(4\log 	Q)$. Using~\cite{BFI} one can 
show that for
sufficiently large $Q$, there is a prime $p$ in the interval $((1-\eps)Q,Q]$ such
that $p-1$ has a divisor $d$ in the interval
$((1-2\eps)\sqrt{Q},(1-\eps)\sqrt{Q}]$. If we write $p-1=df$, then $w_{p} \ge
p-f-d$.  But, if $Q$ is so large that $\eps \le 0.01$, then
$$
f+d = \frac{p-1}{d}+d \le \frac{x}{(1-2\eps)\sqrt{Q}}+(1-\eps)\sqrt{Q}
\le (2+3\eps) \sqrt{p},
$$
which implies the desired result.
\end{proof}

\begin{theorem}
\label{thm:Diam} Let $f(M)$ be any positive function tending
monotonically to
    zero as $M\to \infty$.
Then the  inequality
$$
m-w_{m} \ge m^{1/2} (\log m)^{\kappa/2} (\log\log m)^{3/4} f(m) $$
holds:
\begin{itemize}
\item for  all positive integers $m \le M$, except for possibly $o(M)$ of
them,

\item for all primes $m = p \le M$ except for possibly  $o(M/\log M)$ of
them.
\end{itemize}
\end{theorem}

\begin{proof}
Let $M$ be large and set
$$
z = (\log M)^{\kappa/2} (\log\log M)^{3/4} f(M/2).
$$
It suffices to show $m-w_{m} \le zm^{1/2}$ for $o(M)$ of  integers $m$
    between $M/2$ and $M$.
Without loss of generality, suppose $f(M) \ge 1/\log\log M$ for all
$M>10$. We define $\cJ_k $ to be the set of positive integers $m\in
(M/2,M]$ for which $m-w_{m} \le ym^{1/2}$  and such that 
there are $(x,y) \in \Hb1$ with  $w_m= y-x$ and $x(m-y) = km - 1$.

By the arithmetic-geometric mean inequality, for every
$m \in \cJ_k$, we have 
\begin{equation} \label{eq:LowerBound}
\frac{m-w_{m}}{2}=\frac{m-y+x}{2}\geq\sqrt{x(m-y)}=\sqrt{km-1}.
\end{equation}
Thus $\cJ_k = \varnothing$ for $k \ge z^2+1$.
Suppose $1\le k < z^2+1$,
  $m\in \cJ_k$, $(x,y) \in \Hb1$ and $x(m-y) = km - 1$.
Then $$
\sqrt{kM/2-1} \le \max(x,m-y) \le z\sqrt{M}. $$
By Lemma~\ref{lem:div int},
$$ \# \cJ_k \le
H(M,\sqrt{kM/2-1},z\sqrt{M};\cT_k) \ll \frac{kM(\log
(3z^2/k))^\kappa}
{\varphi(k)(\log M)^{\kappa} (\log\log M)^{3/2}} $$
which after simple calculations leads to the estimate 
$$ \sum_{1\le k < z^2+1} \# \cJ_k  = o(M)
$$
and proves the first part of the theorem.

The proof of the second part is completely analogous.
\end{proof}

The following questions have been formulated as
conjectures in~\cite{FKSY},  where one can also find
some heuristic arguments in their support.

  \begin{question}
Let $g(M)$ be any positive function tending monotonically to
    $\infty$ as $M\to \infty$.
Show that  the inequality
$$
m-w_{m} \le m^{1/2} (\log m)^{\kappa/2} (\log\log m)^{3/4} g(m) $$
holds:
\begin{itemize}
\item for  all positive integers $m \le M$, except for possibly $o(M)$ of
them,

\item for all primes $m = p \le M$ except for possibly  $o(M/\log M)$ of
them. \end{itemize} \end{question}

\begin{question}
\label{quest:dist}
Prove that
$m-w_{m} \le m^{1/2} (\log m)^{\kappa/2+1/2+o(1)}$ for  $m \to \infty$.
\end{question}

Even a weaker form of Question~\ref{quest:dist} with
just $m^{1/2 + o(1)}$ is already of great interest
and would have some interesting applications, see~\cite{KSY}.

Besides the extreme and average values
of the distances between an
element and its modular inverse it is also interesting
to study the number of distinct differences between $x$ and $y$
for $(x,y) \in \Hba$.

\begin{question}
\label{quest:Dist-1}
Estimate the cardinality of the sets
 \begin{eqnarray*}
\cD_{a,m} & = &\{|x-y| \ : \ (x,y) \in \Hba\}, \\
\cD^{\pm}_{a,m}  & = & \{x-y \ : \ (x,y) \in \Hba\}, \\
\cS_{a,m}  & = & \{x+y \ : \ (x,y) \in \Hba\}.
\end{eqnarray*}
\end{question}

It is easy to see that the sets $\cD_{a,m}$ and 
$\cD^{\pm}_{a,m}$ are closely related to each other
due to symmetry of points on $ \Hba$.

Clearly the part of 
Question~\ref{quest:Dist-1} concerning $\cD_{a,m}$ 
is equivalent to
counting $u =0, \ldots, m-1$ for which
the quadratic congruence
$$
x(x+u) \equiv a \pmod m
$$
has a solution $x$ with $1 \le x < m-u$.
In the case of  $\cS_{a,m}$ one needs to count 
$u = 0, \ldots, \pm 2(m-1)$
for which 
$$
x(u-x) \equiv a \pmod m
$$
has a solution $x$ with $1 \le x < u$,

For prime $m=p$ this question has been studied in~\cite{ShpWin1}
where an explicit formula
$$
\# \cD_{a,p} = \frac{1}{4}\(p+1+\(\frac{a}{p}\)
\(1+(-1)^{(p-1)/2}\)\).
$$
is given, where $(a/p)$   denotes the {\it Legendre symbol\/} of $a$
modulo $p$. It has been shown in~\cite{EKSY} that the argument 
of~\cite{ShpWin1} can also be used to derive explicit
formulas for $\cD^{\pm}_{a,p}$ and $\cS_{a,p}$. However 
in the case of composite $m$ studying these quantities is
more complicated and sometimes leads to some surprising 
discoveries. In particular, it is shown 
in~\cite[Theorem~16]{EKSY} that for $a=1$ we have
\begin{equation}
\label{eq:Sum/Diff}
\liminf_{m \to \infty}
\frac{\#\cD^{\pm}_{1,m} \log \log m}{\#\cS_{1,m}} < \infty \mand 
\limsup_{m \to \infty}
\frac{\cD^{\pm}_{1,m}}{\#\cS_{1,m}\log \log m} > 0,
\end{equation}
see also~\cite{Khan2}.
These results also have some interesting combinatorial 
interpretation.

In~\cite{ShpWin1} an analogue of 
Question~\ref{quest:Dist-1} concerning $\cD^{\pm}_{a,p}$ is also studied
when $x$ and $y$ range over prescribed intervals, however the method 
of~\cite{ShpWin1}   does not immediately apply to the case of composite
moduli $m$.

A similar question can also be asked about Euclidean distances
between distinct  points of $\Hba$ and also between the origin
and points of  $\Hba$.

\begin{question}
\label{quest:Dist-2}
Estimate the cardinality of the sets
$$
\cE_{a,m} = \{\sqrt{(x_1-x_2)^2 +(y_1-y_2)^2}  \ : \ (x_1,y_1),  (x_2,y_2) \in
\Hba\}
$$
and
$$
\cF_{a,m} = \{\sqrt{x^2 +y^2}  \ : \ (x, y) \in \Hba\} .
$$
  \end{question}

The author is grateful to Arne Winterhof for the
observation that if $m = p$ is a prime then $\#\cF_{a,p}$ can be evaluated explicitely. 
Indeed
if  $u \equiv  x^2+y^2 \pmod p$ for $(x, y) \in \overline{\cH}_{a,p}$
then  the congruence $Z^4-uZ^2+a^2\equiv 0 \pmod p$ has exactly four
roots $x, p-x, y, p-y$ (in particular, we have  two double roots
if  $x=y$).
Since $(p-x)^2+(p-y)^2\ne x^2+y^2$ we have
$$
\#\cF_{a,p} =  \frac{1}{2}\(p +\(\frac{a}{p}\)\).
$$

Clearly, Theorem~\ref{thm:UD Gen} can be used to obtain some lower
bounds on $\# \cD_{a,m}$ and  $\# \cE_{a,m}$. For example,
an easy modification of an argument which leads to the
bound~\eqref{eq:Large Dist} yields
$$
\# \cD_{a,m} \ge m^{1/4 + o(1)}.
$$
However we are mostly
interested in more precise results which  should
certainly be based  on some additional ideas.

We also ask a question of a different flavour,
which is about the number of possible directions on the Euclidean
plane defined by the pairs of
distinct points $(x_1,y_1),  (x_2,y_2) \in \Hba$.

\begin{question}
\label{quest:Direct}
Estimate the cardinality of the set
$$
\cL_{a,m} = \left\{\frac{x_1-x_2}{y_1-y_2}  \ : \ (x_1,y_1),  (x_2,y_2) \in
\Hba,\ (x_1,y_1) \ne (x_2,y_2)\right\} .
$$
  \end{question}

Obviously Questions~\ref{quest:Dist-2} and~\ref{quest:Direct}
are influenced by the {\it Erd{\H o}s} and
{\it Kakeya\/} problems, respectively,
see~\cite[Sections~5.3 and~7.1]{BMP} and
also surveys~\cite{Bour2,KatTao} (we also note the recent 
remarkable achievement~\cite{Dvir,GuKa}).
They can also be asked for points of  $\cH_{a,m}(\cX,\cY)$
for various sets $\cX$ and $\cY$.

We note that in the definitions of the sets 
$\cE_{a,m}$,  $\cF_{a,m}$ and $\cL_{a,m}$
all distances and angles are computer over the rationals.
The same questions can also be asked with similar quantities 
where calculations are performed modulo $m$. 
Some interesting results in this direction can be
found in~\cite{EKSY,HanKhan,Khan2}.

Finally, motivated by~\cite{Boca2,BoCoZa1} and some 
other works one can also ask various questions
about the distribution of the  angles of elevation 
$\arctan(y/x)$
of points $(x,y) \in \cH_{a,m}(\cX,\cY)$
over the horizontal line.

\subsection{Convex Hull}
\label{sec:Convex}

We consider the convex closure $\cC_{m}$
of the point set $\Hb1$.    It is not hard to see that
$\cC_{m}$  is always a convex polygon with
nonempty interior, except
when $m=2,3,4,6,8,12,24$, see~\cite{KSY}.

Following~\cite{KSY}, we denote by $v(m)$
the number of vertices of $\cC_m$ and
by $V(M)$  its average value,
$$
V(M)=\frac{1}{M-1}\sum_{m=2}^M v(m).
$$


First of all we notice that by a result of~\cite{Andrews}
a convex polygon of area $S$ may have at most $O(S^{1/3})$
integer vertices, thus we immediately conclude
that $v(m) \ll m^{2/3}$. 
It is shown in~\cite{FKS} that a combination of the 
bound of~\cite{Andrews} 
with  Theorem~\ref{thm:UD Gen} leads to a stronger estimate: 
\begin{equation}
\label{eq:Upper Bound 1}
v(m) \le m^{7/12+o(1)}.
\end{equation}
Indeed, one can see from Theorem~\ref{thm:UD Gen} 
that all vertices of $\cC_m$ belong to one of the rectangles 
with one side length $m$ and the other side length $m^{3/4+o(1)}$
which are adjacent  to one of the sides of the square $[0, m-1]\times [0,m-1]$.
Now considering the intersections of $\cC_m$ with  each of these rectangles
(which remains convex) and applying the the result of~\cite{Andrews},
we obtain~\eqref{eq:Upper Bound 1}. The 
bound~\eqref{eq:Upper Bound 1} has recently been improved in~\cite{KonShp}
as 
\begin{equation}
\label{eq:Upper Bound 1.5}
v(m) \le m^{1/2+o(1)}
\end{equation}
for any integer $m\ge 1$ and as 
\begin{equation}
\label{eq:Upper Bound 1.6}
v(m) \le m^{5/12+o(1)}
\end{equation}
for $m$ which are almost squarefree (in particular, for 
almost all $m$). The proof of~\eqref{eq:Upper Bound 1.6} 
is based on some geometric arguments, and thus on Lemmas~\ref{lem:Equiv} 
and~\ref{lem:Pet}, and the bound on the number of
solutions to bivariate quadratic Diophantine equations in a box
given by Lemma~\ref{lem:Quadratic}. 
Furthermore, the same bounds hold for the convex 
hulls of arbitrary hyperbolas $\Hba$.

More interestingly, one can obtain another bound, which is
much better in some cases and relies on  more specific properties
of $\Hb1$
\begin{equation}
\label{eq:Upper Bound 2}
v(m)\le T(m-1) m^{o(1)},
\end{equation}
see~\cite{KSY}, where $T(k)$ is defined in Section~\ref{sec:ANT}.


The lower bound
\begin{equation}
\label{eq:Lower Bound}
v(m) \ge 2(\tau(m-1)-1)
\end{equation}
is also given in~\cite{KSY}. 
Furthermore, it is shown in~\cite{KSY} that
in fact $v(m) = 2(\tau(m-1)-1)$ whenever  $T(m-1) \le 5$.
Thus, this and the bound~\eqref{eq:Upper Bound 2}
explain why Lemma~\ref{lem:div dens} comes into play.
On the other hand, it is shown in~\cite{FKS}
that $v(m) > 2(\tau(m-1)-1)$ for a set of $m$ of positive 
asymptotic density. 
On the other hand, improving the previous estimate of~\cite{FKSY},
it is shown in~\cite{FKS} that~\cite[Theorem~2.1]{AGP} can be used to 
derive that 
$$
v(p+1) \ge \exp\(\(\frac{5\log 2 }{12} + o(1)\) \frac{\log p}{\log 
\log p}\).
$$
for infinitely many  primes $p$. In particular this shows,
in a strong form, that  sometimes $v(m)$ and $\tau(m-1)$
are vastly different orders of magnitude. 

One can also find in~\cite{KSY} several efficient algorithms
for computing $v(m)$, together with their complexity analysis.

Numerical calculations show that while the behaviour
of $v(m)$ is not adequately described by any of the above bounds,
the lower bound~\eqref{eq:Lower Bound} seems to be more precise
than~\eqref{eq:Upper Bound 1} and~\eqref{eq:Upper Bound 2}.

It is quite natural  to view the points of $\Hb1$ as being randomly
distributed in the  square $[0,m]\times [0,m]$ (which is supported by the
theoretic results which we have presented in Sections~\ref{sec:Indiv}
and~\ref{sec:Aver}) and then appeal to the following result 
of~\cite[Satz~1]{ReSu}. Let $\cR$ be a convex polygon in the plane with
$r$ vertices and let
$P_i$, $i=1,\ldots ,n$,  be $n$ points chosen  at random in $\cR$
with uniform distribution. Let $X_n$ be the number of vertices of the
convex closure of the  points $P_i$,  and let $E(X_n)$ be the
expectation of $X_n$. Then
\begin{equation} \label{eq:Expected Number}
                E(X_n) = \frac{2}{3}r(\log n + \gamma) + c_\cR + o(1),
\end{equation}
where $\gamma= 0.577215\ldots$ is the Euler constant,
and $c_\cR$ depends on $\cR$ and is maximal when $\cR$ is a
regular $r$-gon or is affine equivalent to a regular $r$-gon.
In particular, for the
unit square $\cR = [0,1]^2$ we have
$$
c_\cR=  - \frac{8}{3}  \log 2.
$$

Using~\eqref{eq:Expected Number} with $r = 4$, it
seems  plausible to conjecture that for most $m$
$$
v(m)   \approx  h(m),
$$
where
$$
h(m) =  \frac{8}{3}  (\log \varphi(m)  + \gamma  - \log 2 ). $$
However, surprisingly enough, the numerical results of~\cite{KSY}
show that $V(M)$ deviates from
$$
H(M) =  \frac{1}{M-1} \sum_{m=2}^M  h(m) = \frac{8}{3}  (\log M + \gamma + \eta  -1 -
\log 2 + o(1)),
$$
where
$$
\eta =  \sum_{p~\mathrm{prime}}  \frac{\log(1-1/p)}{p} = -0.580058\ldots\,,
$$
quite significantly, and is apparently larger than  $H(M)$
by a fixed factor.  Some partial
explanation to this  phenomenon has been given in~\cite{KSY}
and suggests that for each $m$,
the convex hull $\cC_m$, besides  some ``random'' points,
also contains a
``regular'' component with $2(\tau(m-1)-1)$ points  associated with divisors of
$m-1$ whose average contribution
$$
  \frac{1}{M-1}\sum_{m=2}^M 2(\tau(m-1)-1) \sim 2  \log M
$$
is of the same order of magnitude as the size of $H(M)$.

It is also shown~\cite{KSY} that this affect is specific to the
  points of $\Hb1$ and disappears for the convex hull of points on
a ``generic'' curve which behaves in  much better agreement
with~\eqref{eq:Expected Number}  than $v(m)$.

Clearly since $(1,1), (m-1,m-1) \in \Hb1$ the diameter
of $\Hb1$ takes the largest possible value $\sqrt{2}(m-2)$.
However for the other values of $a$ the question about the diameter
of $\Hba$ is
more interesting.

\begin{question}
\label{quest:Diam}
Estimate the diameter
$$
\Delta_{a,m} = \max\{\sqrt{(x_1-x_2)^2 +(y_1-y_2)^2}  \ : \ (x_1,y_1),
(x_2,y_2)
\in
\Hba\}.
$$
  \end{question}

\subsection{Visible Points}
\label{sec:Visible}

For two sets of integers  $\cX$ and $\cY$ we denote
$$
\cG_{a,m}(\cX,\cY) = \# \{(x,y) \in  \cH_{a,m}(\cX,\cY) \ :\ \gcd(x,y) =1\}
$$
and also, following our usual agreement, we put
$\cG_m(\cX,\cY) = \cG_{1,m}(\cX,\cY)$.

Clearly  $\cG_{a,m}(\cX,\cY)$ is the set of points
$(x,y) \in  \cH_{a,m}(\cX,\cY)$ which are ``visible'' from
the origin (that is, which are not ``blocked'' by
other points with integer coordinates).

The following estimate is obtained in~\cite{Shp2} for
$\cG_m(\cX,\cY)$ but
its extension to the general case is immediate and we present it here
(in fact we also simplify the argument).

\begin{theorem}
\label{thm:Visible}  Let $\cX = \{1, \ldots, X\}$ and $\cY = \{1,
\ldots, Y\}$
where $ 1 \le X,Y \le m$ are arbitrary integers.
For all integers $m$ with $XY \ge m^{3/2}$ we have
$$
\# \cG_{a,m}(\cX,\cY)= \frac{6}{\pi^2} \cdot \frac{XY}{m} \prod_{p\mid m} \(1 +
\frac{1}{p}\)^{-1} + O\(X^{1/2} Y^{1/2} m^{-1/4 + o(1)}\),
$$
where the product is taken over all prime numbers $p \mid m$.
\end{theorem}

\begin{proof} For an integer $z$ we let
$$
\cH_{a,m}(z;\cX,\cY)  = \{(x,y) \in \cH_{a,m}(\cX,\cY) \ : \
z \mid \gcd(x,y)\} .
$$
By the inclusion-exclusion principle, we write
$$
\# \cG_{a,m}(\cX,\cY)= \sum_{z=1}^\infty \mu(z) \# \cH_{a,m}(z;\cX,\cY) .
$$

Clearly
$$
\cH_{a,m}(z;\cX,\cY)  = \emptyset
$$
  if $\gcd(z,m)>1$ or $z > m$.
For $\gcd(z,m) =1$, writing
\begin{equation}
\label{eq:Div d}
x = zs\mand y = zt,
\end{equation}
we have
\begin{eqnarray*}
\lefteqn{\cH_{a,m}(z;\cX,\cY)  =
\{(zs,zt) \ : \ st \equiv az^{-2} \pmod m, }\\
&&  \qquad \qquad\qquad \qquad\qquad \qquad\qquad  1
\le s
\le X/z,\ 1 \le t \le Y/z\}.
\end{eqnarray*}

Now, we define
$$
R =  \rf{X^{1/2} Y^{1/2} m^{-3/4}}\mand
Q = \rf{X^{1/2} Y^{1/2} m^{-1/2}},
$$
and note that
$$
XY/Q^2 \le m.
$$

A variant of Theorem~\ref{thm:UD Gen} gives
$$
\# \cH_{a,m}(z;\cX,\cY)  = \frac{XY \varphi(m)}{z^2 m^2} + O\(m^{1/2 + o(1)}\),
$$
which we apply for ``small'' $z\le R$.

We also note that for each $z$, the product $r = st \le XY/z^2$, where $s$ and
$t$ are given  by~\eqref{eq:Div d}, belongs to a fixed
residue class modulo $m$ and thus can take at
most $XY/z^2m + 1$ possible values.
For each fixed $r \le XY/z^2 \le XY \le m^2$, there are $\tau(r) = r^{o(1)} =
m^{o(1)}$ pairs $(s,t)$  of integers
  $s$ and $t$ with $r = st$,  see~\eqref{eq:tau}.
  Therefore,
$$
\# \cH_{a,m}(z;\cX,\cY) \le \(\frac{XY}{z^2m} + 1\)  m^{o(1)},
$$
which we apply for ``medium'' $z$ with $Q \ge z> R$.

We observe that Lemma~\ref{lem:Exp Sparse} yields the bound
$$
\sum_{\substack{z=1\\ \gcd(z,m) = 1}}^m
\em(A z^{-2} + B z) \ll \gcd(A,B,m)^{1/2} m^{1/2 + o(1)}.
$$
Using the same arguments as
in the proof of Theorem~\ref{thm:UD Gen}, we deduce  that for
the number of positive integers $z \le Z$ with
$az^{-2} \equiv w \pmod m$ with some $w \le W$
is $ZW/m + O( m^{1/2 + o(1)})$.

We now remark that $az^{-2}  \equiv r \pmod m$ where, as before,
$$
r = st \le XY/z^2 \le XY/Q^2 \le m.
$$
Furthermore, for every $z$ the value of $r$ is uniquely defined
and leads to  at most $\tau(r) = m^{o(1)}$ possible
pairs $(s,t)$. Hence,
  the total contribution from ``large'' $z\ge Q$ can be estimated
as
\begin{eqnarray*}
\sum_{m \ge z \ge Q}\# \cH_{a,m}(z;\cX,\cY) &\le &\sum_{\nu =0}^{\rf{2\log m}}
\sum_{2^{\nu +1} Q > z \ge 2^\nu Q}\# \cH_{a,m}(z;\cX,\cY)  \\
&\le & \sum_{\nu =0}^{\rf{2\log m}}
\(2^{\nu +1} Q \cdot\frac{XY}{m(2^{\nu}Q)^2} + m^{1/2}\) m^{o(1)}\\
&\le & \sum_{\nu =0}^{\rf{2\log m}}
\(\frac{XY}{2^{\nu}mQ} + m^{1/2}\) m^{o(1)}\\
& = & \frac{XY}{Q} m^{-1+o(1)}+ m^{1/2 + o(1)} .
\end{eqnarray*}

Combining the above bounds and recalling our choice of $R$ and $Q$
(which optimises the error term),  we derive the desired result. 
\end{proof}

In particular, under the conditions of Theorem~\ref{thm:Visible}
$$
\# \cG_{a,m}(\cX,\cY) \sim
\frac{6}{\pi^2} \cdot \frac{XY}{m} \prod_{p\mid m}
\(1 + \frac{1}{p}\)^{-1}
$$
provided that
$XY \ge m^{3/2+\varepsilon}$ for some fixed $\varepsilon>0$.

It is noticed in~\cite{Chan2} that Theorem~\ref{thm:Visible}
maybe of use for some Diophantine approximation questions, 
see Section~\ref{sec:Approx}. 

A multidimensional version of Theorem~\ref{thm:Visible} has been 
derived in~\cite{ShpWin2}. For the dimension $s=3$ it is based
on Theorem~\ref{thm:UD Aver}. For $s \ge 4$ the 
proof is based on various bounds for multiplicative
character sums in the same style as in~\cite{Shp3,Shp11}.

Visible points satisfying general polynomial congruences have
been studied in~\cite{ShpVol}.  However in~\cite{ShpVol} 
nontrivial results have been obtained only ``on average''
(either over some families of congruences or over some 
moduli). Recently, 
a nontrivial result for  ``individual'' congruences has been
given in~\cite{CGOS}.

Using the technique of~\cite{CGOS}, together with Theorem~\ref{thm:Conc bound}
below, one can probably tackle the following problem: 

\begin{question} Extend Theorem~\ref{thm:Visible} to intervals
of the form
$$
\cX = \{U+1, \ldots, U+X\}\mand \cY = \{V+1, \ldots, V+Y\}
$$
with arbitrary $U$ and $V$.
\end{question}

We also refer to~\cite{MakZah} for further generalisations.

\subsection{Concentration of Points}
\label{sec:Cocentr}

For a positive integer $Z < p$ and arbitrary integers $U$ and $V$,
we denote by $T_{a,p}(Z;U,V)$ the number of
points $(x,y) \in \cH_{a,p}$ which belong to the square $[U+1,U+Z]\times [V+1,V+Z]$.

We remark that Theorem~\ref{thm:UD Gen} implies that $T_{a,p}(Z;U,V) = (1
+o(1)) Z^2/p$ if  $Z\ge p^{3/4+\varepsilon}$  for a 
fixed $\varepsilon >0$, 
and also gives a nontrivial upper bound $T_{a,p}(Z;U,V) =
o(Z)$ if $Z\ge p^{1/2+\varepsilon}$  and $Z = o(p)$ as $p
\to \infty$. These results seem to be the limit of what can be
achieved within the standard exponential sum techniques and
currently available estimates on incomplete Kloosterman sums.

However, in~\cite{CillGar} improving the previous result 
of~\cite{ChanShp}, the following estimate has been given:

\begin{theorem}
\label{thm:Conc bound}
There exists some absolute constant $\eta>0$ such that
for any positive integer $Z < p$, uniformly over
arbitrary integers $U$ and $V$, we have
$$
T_{a,p}(Z;U,V) \le 
\left\{\begin{array}{rll}Z^{4/3+o(1)}p^{-1/3} + Z^{o(1)} &\text{for any $U$ and $V$}, \\
Z^{3/2+o(1)}p^{-1/2} + Z^{o(1)}&\text{if $U = V$}.\end{array}\right.
$$
\end{theorem}

Several more related estimates are given in~\cite{ACZ,BGKS1,BGKS2,CochSih,CoZh,GarGar}.

\subsection{Arithmetic Functions on Points on $\cH_{a,m}$}
\label{sec:Arith Fun}

A modification of the proof of Theorem~\ref{thm:Visible} 
has lead to the asymptotic formula~\eqref{eq:SF prod}. 
There is little doubt that it can also be extended to the case 
where variables are taken intervals of different lengths and also to 
the sums
$$
\sum_{(x,y) \in \cH_m(\cX,\cY)}|\mu(xy)| 
$$
under the same conditions on the sets $\cX$ and $\cY$ as in
Theorem~\ref{thm:Visible}.  

For a prime $p$, asymptotic
formulas for the average values
$$
\sum_{\substack{(x,y)\in \cH_{a,p}(X,Y)\\x \ne
y}}\frac{\varphi(|x-y|)}{|x-y|}\mand
\sum_{\substack{(x,y)\in \cH_{a,p}(X,X)\\x \ne y}} \varphi(|x-y|)
$$
are given in~\cite{Shp7}.

It is also interesting to study sums of other arithmetic
functions on $(x,y) \in \cH_m(\cX,\cY)$.  

\begin{question}
\label{qest:arith sums} For intervals
$$
\cX = \{U+1, \ldots, U+X\}\mand \cY = \{V+1, \ldots, V+Y\}
$$
of length $X,Y \le m$, 
\begin{itemize}
\item obtain nontrivial
bounds for the sums
$$
\sum_{(x,y) \in \cH_{a,m}(\cX,\cY)}\mu(xy) \mand
\sum_{(x,y) \in \cH_{a,m}(\cX,\cY)}\(\frac{x}{y}\),
$$
where $\(x/y\)$ is the {\it Jacobi symbol\/} of $x$ modulo $y$, which we also
extend  to even values of $y$ by simply putting $\(x/y\) = 0$
in this case;
\item obtain estimates or asymptotic formulas for the sums
$$
\sum_{(x,y) \in \cH_{a,m}(\cX,\cY)}\tau(|x-y|) \mand
\sum_{(x,y) \in H_{a,m}(\cX,\cY)} \omega(|x-y|).
$$
\end{itemize}
\end{question}

We remark that in the case when $\cX$ and $\cY$ are very large (compared 
to $m$) intervals of the form
\begin{equation}
\label{eq:diadic}
\cX = \{X+1, \ldots, 2X\}\mand \cY = \{Y+1, \ldots, 2Y\}
\end{equation}
the bound on the sum of Jacobi symbol
$(x/y)$ has been given in~\cite{Yu}. 
In fact, the bound of~\cite{Yu} applies to 
even more general bilinear sums and asserts that for any $\varepsilon > 0$
\begin{equation}
\label{eq:bilin Jacobi}
\sum_{\substack{(x,y) \in \cH_{a,m}(\cX,\cY) \\ 
xy \le Z}}\alpha_x \beta_y \(\frac{x}{y}\)
\ll XY^{15/16 + \varepsilon} +  X^{15/16 + \varepsilon} Y
\end{equation}
provided that 
\begin{equation}
\label{eq:m small 1}
m \le \(\min\{X, Y\}\)^{\varepsilon/3},
\end{equation}
where $\cX$ and $\cY$ are of the form~\eqref{eq:diadic}, 
$Z$ is an arbitrary parameter and the sequences $\alpha_x$ and $\beta_y$
are supported on $\cX$ and $\cY$, respectively, 
and satisfy
$$
|\alpha_x| \le 1, \ x \in \cX\mand 
|\beta_y| \le 1,\ y \in \cY.
$$
Since the sum on the left hand side of~\eqref{eq:bilin Jacobi}
does not exceed $XY/m + \min\{X, Y\}$ we see that the bound 
in nontrivial only if 
\begin{equation}
\label{eq:m small 2}
m \le \(\min\{X, Y\}\)^{1/16 - \varepsilon}.
\end{equation}

Comparing~\eqref{eq:m small 1} and~\eqref{eq:m small 2} we see
that~\eqref{eq:bilin Jacobi} may only be nontrivial if 
$$
 \min\{X, Y\}\ge m^{64}.
$$
Some modifications of the argument of~\cite{Yu} may 
reduce the exponent $64$, but certainly not below 16, 
while in  Question~\ref{qest:arith sums} we ask 
about much shorter sums.

\section{Applications}

\subsection{The Lehmer Problem}
\label{sec:Lehm}

One of the most natural and
immediate applications of the uniformity of distribution
results outlined in Section~\ref{sec:Indiv} is a positive
solution to the {\it Lehmer problem\/},
see~\cite[Problem~F12]{Guy},
about the joint distribution of the parity of $x$ and $y$
for $(x,y) \in   \cH_{a,m}(\cX, \cY)$
with some intervals
$$
\cX = \{U+1, \ldots, U+X\}\mand \cY = \{V+1, \ldots, V+Y\}.
$$
This distribution  is naturally expected to be  close to
uniform (that is, each parity combination is taken about
in 25\% of the cases) for any odd $m\ge 1$ and sufficiently large $X$ and
$Y$.  The Lehmer problem
can easily be reformulated as a question about  the cardinality
of $\#\cH_{a,m}(\widetilde \cX, \widetilde \cY)$ for some
$a$ and some
other sets $\widetilde \cX$ and $\widetilde \cY$ of about $X/2$ and
$Y/2$ consecutive integers, respectively. Indeed, we are interested in
solutions to the congruence
$$
(2\widetilde x+\vartheta_1)(2 \widetilde y+\vartheta_2) \equiv a \pmod m
$$ with some fixed $\vartheta_1, \vartheta_2 \in \{0,1\}$ and
$$
\frac{U+1 - \vartheta_1}{2} \le \widetilde x \le \frac{U+X- \vartheta_1}{2},
\qquad \frac{V+1 - \vartheta_2}{2} \le \widetilde y \le \frac{V+ Y -
\vartheta_2}{2}.
$$
It remains to notice that the above congruence is equivalent to
$$
( \widetilde x+2^{-1}\vartheta_1)(\widetilde y+ 2^{-1}\vartheta_2)
\equiv 4^{-1}a \pmod m
$$
(where the inversion is taken modulo $m$), which after a shift
of variables takes  the desired shape.

Close links between the Lehmer problem and bounds
of Kloosterman sums~\eqref{eq:Kloost} 
has been first observed
in~\cite{Zha1,Zha2}.  The question and the above approach have been
extended  in several directions,
see~\cite{ASZ,CoZa1,CoZa2,Liu,LiuZha1,LiuZha2,LiuZha4,LRS,LuYi0,LuYi0.5,MakZah,
RuSch,XiYi,Xu,XuZha1,YuYi,ZhaXue,Zha3,Zha4,Zha5,Zha5.5,Zha6}
and references therein. In particular, in multidimensional 
generalisations,  instead of the
bound~\eqref{eq:Kloost bound} more general bounds of incomplete or
multiple 
Kloosterman sums 
$$
\sum_{U_1+1\le x_1 \le U_1+X_1} \ldots \sum_{U_n+1\le x_n \le U+n + X_n} 
\em\(r_1x_1 + \ldots + r_nx_n + s(x_1\ldots x_n)^{-1}\),  
$$
have been frequently used.  
However,  it has turned out that for multivariate
analogues of the Lehmer problem, and a  number of similar questions, bounds of
character sums provide a more efficient tool than 
Kloosterman sums~\eqref{eq:Kloost},
see~\cite{Shp11}, where Lemmas~\ref{lem:PVB} and~\ref{lem:4th Moment}
(for prime $m=p$)  and the
bound~\eqref{eq:2nd Moment} have been used to
improve several previous results.

Some ternary additive problems with points $(x,y) \in   \cH_{a,m}$ satisfying 
additional divisibility conditions are considered in~\cite{LuYi1}.
The result of~\cite{LuYi1} has been improved in~\cite{ShpWin3}
where it is shown that one in fact can deal with binary 
additive problems with such numbers. 

A generalisation of the original problem to joint distribution 
of arbitrary monomials and two arbitrary progressions 
is given in~\cite{Shp9},
which improves and generalises some results of~\cite{YiZha}. 

Furthermore, a similar problem has also appeared in a very different context 
related to cryptography. Indeed, it has been shown in~\cite{BCPP}
that the conjecture of~\cite{GorKlap} on  the cyclical distinctness 
of so-called $\ell$-sequences modulo $p$ can be reformulated as the 
conjecture that for any prime $p$ and pair of integers $(a,d) \ne (1,1)$
with
$$
\gcd(d, p-1) =1, \qquad 0<|a|<p/2, \qquad |d|<p/2,
$$
except for six explicitly listed triples of $(p,a,d)$,
at least one residue modulo $p$ of $ax^d$ is odd when $x$ runs 
through all even residues modulo $p$, that is, 
that $ax^d \equiv y \pmod p$ for some  
$$
x \in \{2, 4, \ldots, p-1\} \mand y \in \{1, 3, \ldots, p-2\}. 
$$
In~\cite{GKMS} this conjecture has been established for a wide 
class of triples $(p,d,a)$ and also for all such triples.
Finally, in~\cite{BCPP} it has been proved for all triples 
$(p,d,a)$  provided that $p$ is large enough. 
Clearly $d = -1$ is related to the Lehmer conjecture.
Recently, a decisive progress in a generalised Lehmer conjecture
has been made in~\cite{BCP}.
It is certainly interesting to see how much the methods
of~\cite{BCP,BCPP} can be expanded and carried over to
other problems. For examples, in 
the most general form one can ask about the existence 
and distribution of
solutions to the congruence
$$
x_1^{d_1}\ldots x_s^{d_s} \equiv a \pmod p, \qquad 1 \le x_1, \ldots, x_s <p
$$
in $s$ arithmetic progressions 
$x_i \equiv  b_i \pmod {k_i}$, $i=1, \ldots, s$.

\subsection{Distribution of Angles in Some Point Sets}
\label{sec:Angles}

A version of  Theorem~\ref{thm:UD Gen} has been used 
in~\cite{BoCoZa1} to study the distribution of angles 
between visible points (viewed from the origin) 
in a dilation of a certain plain 
region $\Omega \subseteq [0,1]^2$. More precisely, for 
$\Omega \subseteq [0,1]^2$ and a sufficiently large $Q$, 
we define 
$$
\Omega_Q = \{(Q\alpha, Q\beta)\ : \ (\alpha,  \beta) \in \Omega)\}.
$$
We now consider $ \# \cF_\Omega(Q)$ angles  between the horizontal line
and
the points of the set
$$
\cF_\Omega(Q) = \{(x,y) \in \Omega_Q\cap \Z^2 \ : \ \gcd(x,y) =1 \}. 
$$
The distribution of these angles is studied in~\cite{BoCoZa1} and 
shown to exhibit a somewhat unexpected behaviour. 

A slight generalisation of Theorem~\ref{thm:UD Gen} 
has also played an important role in the study of angles 
defined by some other point sets~\cite{Boca2}, 
see also~\cite{BoZa1,BoZa2,Ust6}.  

\subsection{Distribution  of the Trajectory Length in  
the Periodic Lorentz Gas}
\label{sec:Lonetz}

Links between Theorem~\ref{thm:UD Gen} and some problems of mathematical 
physics can be found in~\cite{BoGoZa,BoZa3,BykUst1,BykUst2}.
 Namely, assume that 
a particle moves along a straight line in a $d$-dimensional lattice
until it falls in a $\delta$-neighbourhood of a lattice point.  
In~\cite{BoGoZa,BoZa3,BykUst1,BykUst2}
the distribution of the trajectory length is studied. 

\subsection{Vertices of Convex Lattice Polygons}
\label{sec:Polygons}

Let $N(r)$ be the number of vertices of the convex hull 
of the set  of
the integral lattice points inside of the ball of radius $R$ 
centered at the origin, that is, of the set
$\{(x,y) \in \Z^2\ : \ x^2+y^2\le r^2\}$. 
It is shown in~\cite{BaDe} that 
for $1 \le H \le R$ we have
$$
\frac{1}{H} \int_{R}^{R+H} N(r)\, dr = \frac{6 \cdot 2^{2/3}}{\pi}
R^{2/3} + O\(HR^{-1/3} + H^{-1} R^{2/3} + R^{5/8 + o(1)}\).
$$
A version of  Theorem~\ref{thm:UD Gen} is one of the crucial ingredients
of the proof.

\subsection{Arithmetic Structure  of Shifted Products}
\label{sec:P(ab+1)}

The main result of~\cite{Stew} on the largest prime divisor 
of the shifted products $ab +1$, when $a$ and $b$ run through 
some reasonably dense sets of integers of a 
sufficiently large interval  $[1,N]$, is based on a variant 
of Theorem~\ref{thm:UD Gen}. The  result of~\cite{Stew} 
has recently been improved in~\cite{Mat}, where it is also 
shown that the above problem is related to counting 
$\# \cH_p(\cX,\cY)$ on average over primes $p$ for arbitrary
sets of integers $\cX, \cY\subseteq \Z$. 

Furthermore, in~\cite{LeBd2}, Theorem~\ref{thm:UD Gen}
and its  generalisations given in~\cite{Shp3,Shp11} (that are based 
on bounds on multiplicative character sums) have been used to study 
$k$th powerfree values of the polynomial $X_1...X_r - 1$, see also~\cite{FoShp2}
for some other related results and alternative approaches to problems
of this type. 

\subsection{Sums of Divisor Functions over 
Quadratic Polynomials and Arithmetic Progressions}
\label{sec:Div Progr}

It has been demonstrated in~\cite{Byk1} that results about the distribution 
of points  $(x,y) \in \cH_{a,m}$ in the regions of the form~\eqref{eq:Curve}
can be used to derive precise asymptotic formulas for sums
of the divisors with quadratic polynomials; for example, for 
the sums
$$
S_D(N) = \sum_{n \le N} \tau(n^2 + D)
$$
where $D\ge 1$ is a squarefree integer. 

Furthermore,  the strength of estimates on
the error terms in asymptotic
formulas for sums of some  divisor  functions over an arithmetic
progression 
$$
S_{a,m}(N) = \sum_{\substack{n \le N\\ n \equiv a \pmod m}} \tau(n)
$$
is closely related to the precision of our knowledge of
$\# \cH_{a,m}(\cX, \cY)$, where $\cX$ and $\cY$ are sets of
consecutive integers and also multidimensional analogues of 
this question,  see~\cite{BHS,Fouv1,FrIw1,FrIw2,HB1,True}. This
link  becomes transparent if one recalls that a standard approach to
evaluating divisor sums is approximating the
   hyperbolic region $\{(x,y)\ : \ x,y \ge 0, \ xy  \le  N\}$ by a
union of ``small'' rectangles.
In fact, as we have mentioned, the proof of Theorem~\ref{thm:UD Aver}
is based on some ideas of~\cite{BHS}, where an asymptotic formula
is given for the sums $S_{a,m}(N)$ ``on average'' over $a$. 
In~\cite{True}, there are also several results, and 
conjectures, for similar average values 
of various restricted versions of the divisor 
functions $\tau(n)$, but with less  averaging.  These 
results are also related to the problem of Section~\ref{sec:PairCorr}.

\subsection{Pair Correlation of Fractional Parts of $\alpha n^2$}
\label{sec:PairCorr}

It is known, see~\cite{MarStr,RudSar},
that  the spacings between  the fractional parts 
of the sequence $\alpha n^2$, $n =1, 2, \ldots$,  
obey a Poisson distribution 
for almost all real $\alpha$. However no specific example
of such $\alpha$ is known. However~\cite[Theorem~2]{HB2} comes
very close to giving such an example. Its proof, amongst other 
technical tools, is also based on tight upper bounds for the divisor 
function over short arithmetic progressions, which in turn 
leads to studying points on modular hyperbolas. In particular,
we note that~\cite[Conjecture~2]{HB2}, can be reformulated as
the upper bound 
$$
\# \cH_{a,m}(\cX,\cY) \ll \frac{\varphi(m)}{m^2} Z^2 
$$
for the sets $\cX = \cY = \{1,\ldots, Z\}$
with an integer $Z \ge m^{1/2 + \varepsilon}$ 
for any fixed $\varepsilon$. Besides, it is shown 
in~\cite{HB2} that one of the approaches to
the distribution of spacings between  the  fractional parts 
of $\alpha n^2$ is via studying the number of solutions 
of the quadratic congruence
$$
x^2-y^2 \equiv c \pmod q, \qquad 1 \le x \le X, \ 1\le y \le Y,
$$
on average over $c$ that runs through the reduced residue classes
modulo $q$. Furthermore, one can find such a bound in~\cite[Lemma~3]{HB2}.
It is shown in~\cite{Shp14} that a slight
generalisation of Theorem~\ref{thm:UD Aver} leads to an 
improvement of~\cite[Lemma~3]{HB2} for some parameter ranges.
However, unfortunately in the range relevant to the pair correlation
problem the corresponding results of~\cite{HB2} and~\cite{Shp14}
are essentially of the same strength. 

Similar links between the distribution of fractional 
parts of $\alpha n^2$, behaviour of the divisor function 
in arithmetic progressions and distribution of points on 
modular hyperbolas have also been exhibited in~\cite{True}. 
It is shown in~\cite{Shp13} that Theorem~\ref{thm:UD Aver} 
can be used to derive improvements on some of the results of~\cite{True}.

\subsection{The Sato--Tate Conjecture in the ``Vertical'' Aspect}
\label{sec:S-T vert}

We also recall that Theorem~\ref{thm:UD Aver}
about the distribution of points on $\cH_{a,m}$
on average over $a$ has been used in studying the
so-called ``vertical'' aspect of the Sato--Tate conjecture
for Kloosterman sums~\eqref{eq:Kloost}, 
see~\cite{Shp6}. The relation is
provided by the identity $K_m(r,s) = K_m(1,rs)$ which holds
if $\gcd(r,m) = 1$.
Clearly for the complex conjugated sum we have
$$
  \overline{K_m(r,s)}  =  K_m(-r,-s) =  K_m(r,s), 
$$
hence we see that   $K_m(r,s)$ is real.
Since by the Weil bound, for any prime $p$, we have
$$
\left| K_p(r,s)\right| \le 2 \sqrt{p}, \qquad \gcd(r,s,p) =1,
$$
see~\cite[Theorem~11.11]{IwKow},
we can now define the angles $\psi_p(r,s)$
by the relations
$$
  K_p(r,s) = 2\sqrt{p} \cos \psi_p(r,s)\mand 0 \le \psi_p(r,s) \le \pi.
$$

The famous {\it  Sato--Tate\/} conjecture asserts that,
in the ``horizontal'' aspect, that is,
for any fixed non-zero integers $r$ and $s$, the angles
$\psi_p(r,s)$ are distributed according to  the
{\it  Sato--Tate  density\/}
$$
  \mu_{ST}(\alpha,\beta) = \frac{2}{\pi}\int_\alpha^\beta  \sin^2 
\gamma\, d \gamma,
$$
see~\cite[Section~21.2]{IwKow}.
More precisely, if $\pi_{r,s}(\alpha, \beta; T)$ denotes the number
of primes $p\le T$ with $\alpha \le \psi_p(r,s) \le \beta$,
the Sato--Tate  conjecture predicts that
\begin{equation}
\pi_{r,s}(\alpha, \beta; T) \sim  \mu_{ST}(\alpha,\beta)  \pi(T), 
\qquad T \to \infty,
\end{equation}
for all fixed real $0 \le \alpha < \beta \le \pi$, 
where, as usual, $\pi(T)$ denotes the total number of primes $p \le T$, 
see~\cite[Section~21.2]{IwKow}.
We remark that for elliptic curves,  the Sato--Tate conjecture
in its original (and the most difficult)
``horizontal'' aspect has recently been settled~\cite{Taylor},
but it still remains open for 
Kloosterman sums~\eqref{eq:Kloost}.
It is shown in~\cite{Shp6} that
Theorem~\ref{thm:UD Aver} implies that
$$
\frac{1}{4RS}\sum_{0 < |r| \le R} \sum_{0< |s| \le S}
\pi_{r,s}(\alpha, \beta; T)   \sim  \mu_{ST}(\alpha,\beta) \pi(T)
$$
provided $RS \ge T^{1 + \eps}$ for some fixed $\eps> 0$.

\subsection{Farey Fractions and Quotients of the Dedekind Eta-function}
\label{sec:Farey}

A result on the average values of some bivariate functions taken over 
points of $\cH_{a,m}$,  has been obtained in~\cite{BoCoZa2} as a tool
in the study of the distribution of Farey fractions
\begin{equation}
\label{eq:Farey}
\cF(T) = \{r/s \in \Q\ : \ \gcd(r,s) = 1, \ 1 \le r <s \le T\}.
\end{equation}
The proof of this result is based on some bounds of incomplete 
Kloosterman sums and 
is based on the same ideas which have led to Theorem~\ref{thm:UD Gen}.

It has also been used~\cite{AXZ} to study the distribution 
of the quotients
of the Dedekind Eta-function $\eta(z)$. 

\subsection{Farey Fractions in Residue Classes and the 
Lang-Trotter Conjecture on Average}
\label{sec:L-T part}

The same arguments  used in the prove of 
Theorem~\ref{thm:UD Aver}, can be used to study 
the distribution of ratios $x/y$, $x\in \cX$, $y \in \cY$,
in residue classes,
where $\cX$ and $\cY$ are sets of the same type as in 
Theorem~\ref{thm:UD Aver}. Combining this with 
an elementary sieve, it has been shown in~\cite{CojShp}
that the set of slightly modified Farey fractions of order $T$, 
that is, the set
$$
\cG(T) = \{r/s \in \Q\ : \ \gcd(r,s) = 1, \ 1 \le r,s \le T\},
$$
is uniformly distributed in residue classes modulo a
prime $p$ provided $T  \ge p^{1/2 +\eps}$ for any fixed $\eps>0$.

In turn, this result has been used in~\cite{CojShp} to improve
upper bounds of~\cite{CojHall}
 for the Lang--Trotter conjectures on Frobenius traces
and Frobenius fields
``on average'' over a one-parametric family of elliptic curves
$$
\E_{A,B}(t): \qquad U^2 = V^3 + A(t)V + B(t)
$$
with some polynomials $A(t),B(t) \in \Z[t]$ when 
the variable $t$ is specialised to the elements of $\cG(T)$. 
A similar result can also be obtained for the set $\cF(T)$ of 
``proper'' Farey fractions given by~\eqref{eq:Farey}.

\subsection{Torsion of Elliptic Curves}
\label{sec:Ell Tors}

A variant of Theorem~\ref{thm:UD Gen}
has been obtained in~\cite{Merel} and applied to estimating
torsion of elliptic curves. More precisely, let  $\cE$
be an elliptic curve over an
algebraic number field $\K$ of degree $d$ over $\Q$.
It is shown in~\cite{Merel} that if  $\cE$
contains a point of prime order $p$ then
$$
p \le d^{3d^2}.
$$

\subsection{Ranks and Selmer Groups of Elliptic Curves}
\label{sec:Ell Ranks}

In~\cite{Yu} the bound~\eqref{eq:bilin Jacobi}
is used to show that for any elliptic curve 
$\cE$ of the form 
$$
\cE: \quad Y^2 = X(X^2 + aX + b) 
$$
where $a$ and $b$ are integers satisfying a certain 
special relation, the sequence of its {\it quadratic twists\/}
$$
\cE_D: \quad DY^2 = X(X^2 + aX + b) 
$$
is of rank $0$ for a positive proportion of squarefree integers $D$. 
We also note that the bound~\eqref{eq:bilin Jacobi}
plays an important role in studying the distribution 
of Selmer groups of similar families of curves, see~\cite{XioZah}.

\subsection{Manin Conjecture for Some del Pezzo Surfaces}
\label{sec:del Pezzo}

Theorem~\ref{thm:UD Gen} has also been obtained in~\cite[Lemma~1]{LeBd1}
where is used a tool to derive an asymptotic formula for 
rational points of bounded height on some del Pezzo surfaces.
A 3-dimensional version of Theorem~\ref{thm:UD Gen} from~\cite{FrIw1}
has been used in~\cite{LeBd3} for a similar purpose. 
These asymptotic formulas correspond to the Manin conjecture 
for these type of varieties.

%
%
%
%

\subsection{Frobenius Numbers}
\label{sec:Frob Numb}

Given a vector $\vec{a}=(a_1, \ldots, a_s)$ of $s$ positive 
integers with $\gcd(a_1, \ldots, a_s)=1$ we define its 
{\it Frobenius Number\/} $F\(\vec{a}\)$ as
the smallest integer $f_0$ such that any integer $f > f_0$
can be represented as
$$
f = a_1 x_1+ \ldots + a_sx_s
$$
with nonnegative integers $x_1, \ldots, x_s$, see~\cite{RaAl} for the background.

It has been demonstrated in~\cite{Sch-Puc,SSU,Ust5} that for the case $s=3$ 
results of the type of Theorem~\ref{thm:UD Gen} become important 
in studying the distribution of the values of $F\(\vec{a}\)$. 

We remark that in the case of $s=3$ it has been conjectured in~\cite{BEZ}
that 
\begin{equation}
\label{eq:BEZ-conj}
F\((a,b,c)\)\le (abc)^{5/8}
\end{equation}
for all vectors $(a,b,c)$ except for some explicitly
excluded family of vectors; see also~\cite[Conjecture~1]{BEZ}
for an even stronger conjecture. 
However the conjectured inequality~\eqref{eq:BEZ-conj} has been disproved 
in~\cite{Sch-Puc}, where it is shown
that for any $a$ there are $\varphi(a)/2$ pairs $(b,c)$ 
with $a< b < c < 2a$  and $\gcd(a,b,c)=1$ for which
$$
F\((a,b,c)\) \ge \frac{a^2}{2}. 
$$
Since $abc \le 4a^3$, this means that the exponent in $5/8$ in~\eqref{eq:BEZ-conj}
has to be increased up to at least $2/3$. 
Since these vectors  are {\it admissible\/} 
in the terminology of~\cite{BEZ} it shows that~\cite[Conjecture~1]{BEZ} fails too.

However one can 
do even better by using Theorem~\ref{thm:UD Gen} and taking 
any pair  $(b,c)$ with $bc \equiv 1 \pmod a$ and $b,c = a^{3/4 + o(1)}$
in the argument of the proof of~\cite[Theorem~1]{Sch-Puc}, 
getting
$$
F\((a,b,c)\) \ge a^{7/4+ o(1)} 
$$
for such vectors $(a,b,c)$. Thus, since $abc= a^{5/2+ o(1)}$, 
we conclude that for any $a$ there are many pairs $(b,c)$  
for which $5/8$ in the conjecture~\eqref{eq:BEZ-conj}
has to be increased up to  $7/10$. Clearly these examples
are quite sporadic and certainly do not belong to the family 
of excluded triples from~\cite{BEZ}.
We also remark that 
the family of  integral vectors of the form 
$$
(a,b,c) = (df -1, d, f) , \qquad d \le f \le d^{1+o(1)},
$$
as $d\to \infty$, 
for which  $abc = a^{2+o(1)}$ leads to the inequality 
$$
F\((a,b,c)\) \ge ab^{1+o(1)} =  (abc)^{3/4+o(1)}.
$$
Furthermore, 
using lower bounds on $H(x,y,z;\N)$, see~\cite{Ford,HalTen}
one can see that the sequence of $a=df-1$ of the above form
is quite dense. 

In the opposite direction, it is shown in~\cite{Ust9}
that for any $a$ and almost all pairs of positive 
integers $(b,c)$ 
with $\max\{b,c\} \le a$ we have 
$$
F\((a,b,c)\) \le   (abc)^{1/2}.
$$

Several other recent results   can  
be found in~\cite{AlGru,AlHeHi,BourSin,FuRob,Sun}.

\subsection{Distribution of Solutions of Linear Equations}
\label{sec:Lin Eq}

The paper~\cite{DiSi}  has  
introduced the question on the distribution of ratios $x/m$ associated 
with the
smallest positive solutions $(x,y)$ to linear equations
$kx - my =1$, on average over the coefficients 
$k$ and $m$ in some  intervals. 
In~\cite{DiSi} it has been studied via continued 
fractions, see also~\cite{Dolg}.
Then, it has been noticed in~\cite{Fuji,Rieg}
that results of the type of Theorem~\ref{thm:UD Gen}
provide a simpler and more direct way to investigate 
the above ratios. Indeed, it is easy  to see that 
this is equivalent to studying the distribution of modular 
inverses  $k^{-1} \pmod m$ which explains the link with 
Theorem~\ref{thm:UD Gen} and of course, 
with Kloosterman sums~\eqref{eq:Kloost}. 
However this argument does not 
take any advantage of averaging over $m$. It is very plausible 
that the results about sums of Kloosterman sums, 
which date back to~\cite{Kuz} can be useful for this 
kind of problem. 

\begin{question} Improve the error terms in the asymptotic formulas
of~\cite{DiSi,Dolg,Fuji,Rieg} by using the results of~\cite{Kuz} 
on cancellations in sums of Kloosterman sums, see also~\cite[Chapter 16]{IwKow}
for the guide to the modern results. 
\end{question}

In~\cite{Shp12}, using some bounds of certain bilinear exponential
sums from~\cite{DuFrIw},  these results have been extended to the case when 
the coefficients $k$ and $m$ run through arbitrary 
but sufficiently dense sets of integers. 

\subsection{Coefficients of Cyclotomic Polynomials}
\label{sec:Cyclo}

Let 
$$
A(n) = \max_{k =0 \ldots, \varphi(n)} |a_n(k)|
$$
be the height of the $n$th cyclotomic polynomial:
$$
\Phi_n(Z) = \sum_{k=0}^{\varphi(n)} a_n(k) Z^k.
$$
In~\cite{GaMo}, Theorem~\ref{thm:UD Gen} has been used to show that for 
any $\varepsilon>0$ and any sufficiently large prime $p$ there exist infinitely many 
pairs $(q,r)$ of distinct primes such that 
$$
A(pqr) \ge \(\frac{2}{3} - \varepsilon\)p,
$$
which disproves a conjecture from~\cite{Bet}.

\subsection{Approximations by Sums of Several Rationals}
\label{sec:Approx}

Following~\cite{Chan2,Chan3}, we consider the problem of obtaining
an upper bound on the approximation of a real $\alpha$ by $s$ rational
fractions with denominators at most $Q$, that is for
$$
\delta_{\alpha,s}(Q) = \min_{1\le q_1, \ldots, q_s \le Q}\left| \alpha -
\frac{r_1}{q_1} - \ldots -
\frac{r_s}{q_s}
\right|
$$
with positive integers  $q_1, \ldots, q_s \le Q$.
The question is motivated by the Dirichlet theorem on
rational approximations
which corresponds to the case $s=1$.

It is shown in~\cite{Chan2,Chan3} that  the results on the distribution
of points on $\cH_{a,m}$ and its multivariate analogues 
are directly related to this
question and imply nontrivial bounds on $\delta_{\alpha,s}(Q)$.

In particular, it is remarked in~\cite{Chan2} that for $s=2$ 
the question of Section~\ref{sec:Visible} on 
visible points on modular hyperbolas becomes relevant
and the result of~\cite{Shp2} implies~\cite[Conjecture~5]{Chan2} 
with any $\vartheta > 3/4$.

Furthermore, it is shown in~\cite{Shp8} that in some cases, 
using Theorem~\ref{thm:UD Aver}, one can improve some of 
the results of~\cite{Chan3}.

\subsection{$\SL_2(\F_p)$ Matrices of Bounded Height}
\label{sec:SL mat}

For a finite field  $\F_p$ of $p$  elements
and a positive integer $T \le (p-1)/2$, we use
  $N_p(T)$ to denote the number of matrices
$$
\begin{pmatrix}  u & v\\ x & y
\end{pmatrix} \in \SL_2(\F_p)
$$
with $|u|,|v|,|x|, |y| \le T$
(we assume that $\F_p$ is represented by
  the elements of the  set $\{0, \pm 1, \ldots, \pm (p-1)/2\}$).

Clearly
$$
N_p(T) = \sum_{a \in \F_p} \# \cH_{a+1,p}(\cT,\cT)
  \#\cH_{a,p}(\cT,\cT)
$$
where $\cT = \{0, \pm 1, \ldots, \pm T\} $ and also
$\# \cH_{0,p}(\cT,\cT) = 4T + 1$.

It has been shown in~\cite{AhmShp} that using the identity
$$
\sum_{a \in \F_p}   \# \cH_{a,p}(\cT,\cT) = (2T+1)^2
$$
and Theorem~\ref{thm:UD Aver}, one can derive that
$$
N_p(T) = \frac{(2T+1)^4}{p} + O\(T^2p^{o(1)}\), 
$$
which is nontrivial if $T \ge p^{1/2 + \varepsilon}$
for any fixed $\varepsilon>0$ and sufficiently large $p$.
More general results on the distribution of 
determinants of  $n$-dimensional matrices with 
entries from an arbitrary (but sufficiently large)
set $\cA \subseteq \F_p$ are given in~\cite{Vinh1},
see also~\cite{Vinh2}. 
We remark that these are   $\F_p$ analogues of the results of similar spirit
for matrices over $\Z$ and algebraic number fields,
see~\cite{DRS,Roet} and references therein. Furthermore, 
if one is only interested in the existence of $\SL_2(\F_p)$ matrices 
with bounded entries then one can apply the results 
from~\cite{AyyCoch} about the existence of solutions to
congruences 
$$
a xv +b yu \equiv c \pmod m
$$
in prescribed intervals.

\subsection{Matrix Products and Continued Fractions}
  \label{sec:Mat Prod}

A variant of Theorem~\ref{thm:UD Gen} plays an important role
in obtaining an asymptotic formula for the number of products of the matrices
$$
\left[\begin{matrix} 1&1\\0&1  \end{matrix}\right]
 \mand \left[\begin{matrix} 0&1\\1&1\end{matrix}\right]
$$
which are of trace at most $N$, see~\cite{Boca1}. In turn, this question
is related to studying the number of reduced quadratic irrationalities
whose continued fraction expansions are of period
length at most $L$. 

Theorem~\ref{thm:UD Gen} has also been used in~\cite{Ust1} 
to estimate a certain  weighted sum over points of $\Hb1$
which appears in studying statistical properties of 
continued fractions of  rational numbers from 
a certain set such as $a/b$ with $a,b\ge 1$ and $a^2 +b^2 \le R$, 
see also~\cite{Avd,Byk2,Ust2,Ust3,Ust7,Ust8}
for other versions and applications of Theorem~\ref{thm:UD Gen}. 

In~\cite{Mosh2}, one can find several results about 
rational fractions $a/m$ with 
partial quotients bounded by a certain parameter $k$.  
The approach of~\cite{Mosh2}
rests on  
a series of very interesting results about the distribution 
of points of $\cH_{a,m}$.

\subsection{Computing Discrete Logarithms and Factoring}
  \label{sec:Dlog Fact}

It has been discovered in~\cite{Sem2} that the distribution of
points on  hyperbolas $\cH_{a,m}$ has direct links with
analysis of some discrete logarithm algorithms. Namely, analysis
of the algorithm from~\cite{Sem2} rests on a weaker version
of Theorem~\ref{thm:UD Aver} obtained in~\cite{Sem1}.

\begin{question} Study whether Theorem~\ref{thm:UD Aver}
can be used to improve
some  of the results of~\cite{Sem2}.
\end{question}

A new deterministic integer factorisation algorithm has recently been
suggested in~\cite{Rub}. Its analysis depends on a variant of
Theorem~\ref{thm:UD Gen}.

\section{Concluding Remarks}

\subsection{Multivariate Gerneralisations}

Multidimensional variants of the above problems have also been
considered and in principle one can use bounds
of multidimensional Kloosterman sums~\eqref{eq:Kloost}
to study the distribution of solutions to the congruence
\begin{equation}
\label{eq:s dim}
x_1 \ldots x_s \equiv a \pmod m
\end{equation}
in the same fashion as in the case of two variables.
However, it has been shown in~\cite{Ayy} that in many cases bounds of
multiplicative character sums lead 
to much stronger results.
For example, in~\cite{Shp11}, using the bounds of Lemma~\ref{lem:PVB}, 
a previous less general version of Lemma~\ref{lem:4th Moment} and the bound~\eqref{eq:2nd Moment}, an improvement
is given of some results of~\cite{ASZ} on the generalised Lehmer problem,
see also~\cite{YuYi}.
Similar ideas have also led in~\cite{Shp3} to some other results on
distribution of solutions to~\eqref{eq:s dim}. For instance,
let $\Delta_{s,a,m}$ denote the discrepancy
of the $s$-dimensional points
$$
\(\frac{x_1}{m},\ldots, \frac{x_s}{m}\) \in [0,1]^s, \qquad x_1 
\ldots x_s \equiv
a \pmod m,
\ 1 \le x_1, \ldots, x_s \le m.
$$
It is shown in~\cite{Shp3} that
$$
\Delta_{s,a,m} \le
\left\{\begin{array}{rll}m^{-1/2 + o(1)} &\text{if}\ s=3, \\
m^{- 1 + o(1)} &\text{if}\ s \ge 4.\end{array}\right.
$$
As we have mentioned, a multidimensional analogue of 
Theorem~\ref{thm:Visible} is given in~\cite{ShpWin2}. 
Using Lemma~\ref{lem:4th Moment}
in its present form one can easily generalise the 
above results. 

We now show how multiplicative characters can be used on 
the example of  counting  the number of solutions to the congruence
$$
x_1 \ldots x_s \equiv
a \pmod m,
\qquad 1 \le x_i \le X_i, \ i =1\ldots, s,
$$
for some  integers $X_1, \ldots, X_s\ge 1$, 
which we denote as $N_s(a,m; X_1, \ldots, X_s)$.

Assuming that $\gcd(a,m)=1$ and
using the identity~\eqref{eq:Ident mult char}, we write
$$
N_s(a,m; X_1, \ldots, X_s) =
\sum_{x_1 = 1}^{X_1} \ldots \sum_{x_s = 1}^{X_s} 
\frac{1}{\varphi(m)} \sum_{\chi \in \varPhi_m}
\chi\(x_1 \ldots x_s a^{-1}\).
$$
Changing the order of summation and separating the contribution 
from the principal character $\chi_0$, which gives the main 
term $X_1 \ldots X_s/\varphi(m)$, we obtain
$$
\left|N_s(a,m; X_1, \ldots, X_s) -
\frac{X_1 \ldots X_s}{\varphi(m)}\right|
\le \frac{1}{\varphi(m)}
\sum_{\substack{\chi \in \varPhi_m\\ \chi \ne \chi_0}}
\prod_{i=1}^s\left| \sum_{x_i = 1}^{X_i} \chi\(x_i\)\right|.
$$

We now assume that $s\ge 4$. Then  we apply Lemma~\ref{lem:PVB}
to $s-4$ sums in the above (say, for $i=5,\ldots, s$)
and then use the H{\"o}lder inequality. This leads to the estimate
\begin{eqnarray*}
\lefteqn{\left|N_s(a,m; X_1, \ldots, X_s) -
\frac{X_1 \ldots X_s}{\varphi(m)}\right|}\\
& & \quad \le \frac{1}{\varphi(m)}  m^{(s-4)(\nu+1)/4\nu^2 + o(1)}
\prod_{i=5}^s X_i^{1 -1/\nu}   \prod_{i=1}^4
\( \sum_{\substack{\chi \in \varPhi_m\\ \chi \ne \chi_0}}
\left| \sum_{x_i = 1}^{X_i} \chi\(x_i\)\right|^4\)^{1/4}.
\end{eqnarray*}
We now apply Lemma~\ref{lem:4th Moment} to the product 
of 4th powers. Certainly depending on the relative 
sizes of $X_1, \ldots, X_s$, Lemma~\ref{lem:PVB} can
give better results if applied to different sums and
maybe with different values of $\nu$ for each sum.

In~\cite{Gar3,Gar6} the above approach based on using 
multiplicative character sums instead of exponential
sums has been complemented  with a very interesting new 
ingredient, which is based on the results and ideas
of~\cite{FrIw2,Hux}.
In particular, it is shown in~\cite{Gar6} that
$$N_3(a,m; X_1, X_2, X_3) > 0
$$
and  if $m$ is cubefree then 
aslo 
$$N_4(a,m; X_1, X_2, X_3, X_4) > 0
$$
for all but $o(m)$ residue
classes modulo $m$, provided that 
$X_1X_2X_3>m^{1 + \varepsilon}$ and $X_1X_2X_3X_4>m^{1 + \varepsilon}$, 
respectively, for some fixed $\varepsilon>0$. This line of research is
very new and promising. 

We, however, note that  even in the multidimensional case, 
multiplicative character sum techniques do not always
win over the  exponential sum approach. For example, 
it seems that for the results of~\cite{Chan5} 
the bound on Kloosterman 
sums~\eqref{eq:Kloost} gives  a more powerful
and appropriate  tool than bounds of  multiplicative character sums. 

Upper bounds on the number of solutions of the congruence~\eqref{eq:s dim}
with variables from short intervals are given in~\cite{BGKS1,CillGar}. 
For example, it is shown in~\cite{BGKS1} that for
a prime $m = p$  and a positive integer
$$
h<p^{1/(s^2-1)}
$$
there are at most  $h^{o(1)}$ solutions in interval of length $h$, 
that is with $x_i \in [k_i+h]$, for some integers $k_i$,
$i =1, \ldots, s$, see also~\cite{BGKS2}.

It is worth noticing that one has to be  careful
with posing multidimensional generalisations, which sometimes lead
to rather simple questions  for $s \ge 3$ (despite that for $s=2$ they
are nontrivial at all).  For example, it has been
noticed in~\cite{Shp3} that often  such generalisations do  not need any
analytic technique used in~\cite{ZhaZha} but follow  immediately
from a very elementary argument. 

\subsection{Gerneralisations to Other Congruences and
Algebraic Domains}

We believe  that it is interesting to explore how much of the theory
developed for modular hyperbolas can be extended to modular circles
$$
\cC_{a,m} = \{(x,y) \ : \  x^2 + y^2 \equiv a \pmod m\}.
$$
Well-known parallels between the properties
of integer points on hyperbolas and circles on the Euclidean plane,
suggest that many of the results obtained for $\cH_{a,m}$
can be extended to $\cC_{a,m}$. We also recall that~\cite[Lemma~3]{HB2}
gives a version of Theorem~\ref{thm:UD Aver} for the 
number of solutions to the congruence
$$
x^2 - y^2 \equiv a \pmod m, \qquad 1 \le x,y\le X,
$$
on average over $a$.

One can also consider the analogue  of the question of 
Section~\ref{sec:SL mat} for higher dimensional matrices.
For example, for it is shown in~\cite{AhmShp} that for 
$T\ge p^{3/4 + \varepsilon}$ the number of matrices
$$
\(x_{ij}\)_{i,j=1}^n\in \SL_2(\F_p)
$$
with $|x_{ij}| \le T$, $i,j =1,\ldots, n$, is 
asymptotic to its expected value $T^{n^2}/p$. This result is based on
different arguments which rely on the results of~\cite{Fouv2,FoKa} 
(and in fact apply to a wide class of polynomial equations).
 
\begin{question}  Close the 
gap between the thresholds $T\ge p^{1/2 + \varepsilon}$ for $n =2$
and $T\ge p^{3/4 + \varepsilon}$ for $n \ge 3$.
\end{question}

In~\cite{HuLi}, using bounds of~\cite{FHLOS}
for exponential sums with matrices, the join distribution
of elements of a matrix $A \in \GL_n(\F_p)$ and its inverse $A^{-1}$
has been studied. In particular, in~\cite{HuLi} some analogues
of the result of~\cite{BeKh} have been derived. 

One of the interesting and still  unexplored lines of research is 
studying 
function field generalisations, which   conceivably should
admit results of the same strength or maybe even stronger as
in the case of residue rings.
For example,  a polynomial analogue of the
conjecture of~\cite{EOS} could be more accessible.

\begin{question} Let $\F_q$ be a
finite field of $q$ elements.
Given an irreducible polynomial $F(X)\in \F_q[X]$ of
sufficiently large degree $d$, show that for any
polynomial $A(X) \in \F_q[X]$, relatively prime to $F(X)$,
there are two irreducible polynomials $G(X),H(X) \in \F_q[X]$
of degree at most $d$ such that
$$
G(X) H(X) \equiv A(X) \pmod{F(X)}.
$$
\end{question}

It is also natural to ask about possible generalisations of 
Theorems~\ref{thm:UD Gen} and~\ref{thm:UD Aver} to matrix equations. 

\begin{question}  Obtain asymptotic formulas, individually for every 
$n\times n$ matrix $A$ over $\Z/m\Z$ and also on average over all
such matrices, for the number of solutions of the congruence
$$
XY \equiv A \pmod m
$$
with matrices $X = \(x_{ij}\)_{i,j=1}^n$ and $Y = \(y_{ij}\)_{i,j=1}^n$
where  $x_{ij} \in \cX_{ij}$, $y_{ij} \in \cY_{ij}$ for some ``interesting'' sets 
$\cX_{ij},\cY_{ij} \subseteq \Z/m\Z$, $i,j =1,\ldots, n$ (for example,
when elements belong to prescribed short intervals).
\end{question}

\subsection{Ratios Instead of Products}

Questions about the distribution of points on 
modular hyperbolas can be reformulated as questions 
about the distribution of products $xy$ in residue classes
modulo $m$. In turn this naturally leads to 
similar  questions about the
distribution of ratios $x/y$ (with $\gcd(y,m)=1$)
in residue classes. Although typographically similar, 
many aspects of these  new questions are very different.

For example, we have already mentioned in Section~\ref{sec:Indiv}
that for any prime $p$,  any  integer $a$ 
can easily be shown to 
be represented as $x/y \equiv a \pmod p$ for some integers $x$ and $y$ with
$|x|,|y|\le p^{1/2}$, while for the products this statement is
not correct and even obtaining a relaxed statements with
$|x|,|y|\le p^{\gamma}$ for some  $\gamma < 3/4$ 
is still an open problem (and apparently is very hard).

Furthermore, by Theorem~\ref{thm:UD Gen}, every integer $a$
can be represented in the form  $xy \equiv a \pmod p$ for some integers $x$ and $y$ with
$1 \le x,y\le p^{3/4 +\varepsilon}$ for any $\varepsilon$ and sufficiently
large $p$, while obviously the congruence 
$x/y \equiv -1 \pmod p$ has no solution with $1 \le x,y\le p/2$.

On the other hand, in~\cite{Shp6}, 
a  full analogue of  Theorem~\ref{thm:UD Aver} 
is given for the ratios $x/y$ as well, 
see also~\cite{CojShp,Gar1,GarKar2,Sem1}
and Section~\ref{sec:L-T part}.

Further investigation of distinctions and similarities between
these two groups of questions is a very interesting direction of 
research.

\subsection{Further Perspectives}

Here we mention some results which have not been
used in this area, which we believe can lead 
to some new directions of research.

We have seen that bounds of Kloosterman 
sums~\eqref{eq:Kloost} 
and several other celebrated number theoretic results and techniques, 
play a
prominent role in this field. Still, it is highly important to
further extend the scope of number theoretic tools which 
can be used for studying the points on modular hyperbolas and
their generalisations.  In particular, it would be  very 
interesting to find 
new   applications of the
bounds of very short
incomplete Kloosterman sums 
from~\cite{Bour1,Kar1,Kar2,Kor,Luo,Shp4,WaLi}.

Furthermore, one may expect that the bounds of bilinear sum
$$
\sum_{M \le m \le 2M} 
\sum_{X \le x \le 2X} \alpha_m \vartheta_x \em(a x^{-1}), \qquad a \in \Z,
$$
from~\cite{DuFrIw} (where  $(\alpha_m)$ and  $(\vartheta_x)$
are arbitrary sequences supported on the intervals $[M,2M]$ 
and $[X, 2X]$, respectively) can be useful for studying the 
points on $\cH_{a,m}$ in very general sets on 
average over moduli $m$ taken from another general set. 
One of such applications has been mentioned in Section~\ref{sec:Lin Eq}.

In the same spirit, finding   applications of the bounds of sums of 
Kloosterman sums~\eqref{eq:Kloost} 
to the study of points on $\cH_{a,m}$  would be of great interest,
see~\cite[Chapter~16]{IwKow} for a background on such results.

\end{document}